\documentclass[11pt,a4paper,reqno]{amsproc}
\usepackage{amsmath, amsthm, amscd, amsfonts, amssymb, graphicx, color}
\usepackage{subcaption}
\usepackage{enumerate}
\usepackage{lmodern}
\usepackage[bookmarksnumbered, colorlinks, plainpages]{hyperref}
\hypersetup{colorlinks=true,linkcolor=red, anchorcolor=green, citecolor=cyan, urlcolor=red, filecolor=magenta, pdftoolbar=true}

\newtheorem{theorem}{Theorem}[section]
\newtheorem{lemma}[theorem]{Lemma}
\newtheorem{proposition}[theorem]{Proposition}

\newtheorem{Remark}[theorem]{Remark}
\theoremstyle{definition}
\newtheorem{definition}[theorem]{Definition}

\theoremstyle{remark}

\newtheorem{note}[theorem]{Note}

\numberwithin{equation}{section}

\usepackage[square,numbers]{natbib}

\begin{document}

\setcounter{page}{1}

\title[Bi-parametric potential operators]
 {Wavelet-based inversion and analysis of Flett, Riesz and bi-parametric potentials in $(k,1)$ generalized Fourier framework}

\author[Athulya P, Umamaheswari S, S.K. Verma]{
  Athulya P\textsuperscript{a}, Umamaheswari S\textsuperscript{a}, \and Sandeep Kumar Verma\textsuperscript{a} \\
  \textsuperscript{a} Department of Mathematics, SRM University AP, Amaravati 522240, Andhra Pradesh, India  \\
Correspondence to : \texttt{sandeep16.iitism@gmail.com}}



\begin{abstract}
In this paper, we construct and analyze Bessel and Flett potentials associated with the heat and Poisson semigroups in the framework of the $(k,1)$-generalized Fourier transform. We establish fundamental properties of these potentials and derive an explicit inversion formula for the Flett potential using a wavelet-like transform. Furthermore, we introduce a $\beta$-semigroup $\mathcal{B}_k^{(\beta,t)}$, defined via $W_k^{(\beta, t)}$, which enables the formulation of an inversion formula for the Riesz potential. As a unifying extension, we define and investigate bi-parametric potentials $\mathfrak{J}_k^{(\alpha,\beta)}$, which generalize both the Bessel potential and the Flett potential. In addition, we define the associated function spaces. 
\end{abstract}

\maketitle

\section{Introduction}
Potential theory is a branch of mathematical physics that originated in the 19th century. It primarily deals with the study of harmonic functions, which are solutions to Laplace’s equation $(\Delta u=0)$ \cite{Donovan}.  
Potential theory is highly valued in mathematics for its role as a smoothing operator. It transforms rough or irregular functions into smoother, more regular forms, which is especially useful when solving partial differential equations. A classic example is the Riesz potential, a fractional integral operator that generalizes the classical Newtonian potential \cite{Admas}. In Euclidean space, the Riesz potential is defined as $\mathbf{R}_{\alpha}= (-\Delta)^{-\alpha/2},$ where $0< \alpha< n$. The Riesz kernel behaves well locally (as $|x|\rightarrow 0$), making it effective for smoothing. However, its global behavior (as $|x|\rightarrow \infty$) becomes less favorable as
$\alpha$ increases. The Riesz potential can be modified by preserving its local smoothing properties to address this issue. One of the simplest and most natural modifications is to replace the non-negative operator  $(-\Delta)$ with the strictly positive operator  $I-\Delta,$ where $I$ is the identity operator, leading to the Bessel potential $\mathbf{I}_{\alpha}= (I-\Delta)^{-\alpha/2}$, for $\alpha>0$.
Both $(-\Delta)^{-\alpha/2}$, and $(I-\Delta)^{-\alpha/2}$ describe fractional powers of differential operators, though they reduce to Fourier multipliers with the factors $\|x\|^{-\alpha}$ and   $(1+\|x\|^2)^{-\alpha/2}$ respectively \cite{Stein}. Meanwhile, Flett \cite{Flett} introduces a novel fractional integro-differential operator by using
$(I+\sqrt{-\Delta})^{-\alpha}$ in place of the more traditional $(-\Delta)^{-\alpha/2}$ or $(I-\Delta)^{-\alpha/2}$. This approach defines the Flett potential as a Fourier multiplier operator, characterized by the factor $(1+\|x\|)^{-\alpha}$. It provides better localization properties and can be more flexible in handling anisotropic spaces and specific weight functions \cite{Yildiz}.

The theory of potentials was originally developed in the Euclidean context, where it aligned naturally with the classical Laplacian and the standard Fourier transform. Advancements in integral transforms, such as the Dunkl and Fourier-Bessel transforms, have paved the way for the development of new differential operators. These innovations have produced alternative forms of the Laplacian associated with these transforms. As a result, researchers have been motivated to generalize the concept of potentials across diverse frameworks, including differential-difference operators \cite{Ben, Xu, Verma-25} and Laplace-Bessel differential operators \cite{Yildiz}. Moreover, potentials have been explored in more abstract contexts, such as metric measure spaces \cite{Hu}, reflecting the growing interest in generalizing classical analysis and Lebedev–Skalskaya transform \cite{Akhilesh-23} etc.

In 2021, Ivanov \cite{Ivanov1} made a contribution by establishing the Riesz potential in the $(k,1)$-generalized Fourier setting, which involves the operator $\|x\|\Delta_k$, where $\Delta_k$ denotes the Dunkl Laplacian and in 2022, Ben Sa\"id and  Negzaoui \cite{Negzaoui}
proposed a unified theory for the Flett potential within the framework of the $(k, a)$-generalized Fourier transform on the real line. In our work, we develop  Flett and Bessel potentials, as well as introduce a more general representation of the Riesz potential,  associated with the  $(k,1)$-generalized Fourier transform in $\mathbb{R}^n$. 

The $(k, a)$-generalized Fourier transform is the far-reaching generalization of the Fourier transform by introducing two parameters $k$ and $a$. In \cite{orsted, Kobayashi},  Ben Sa\"id, Kobayashi and {\O}rsted constructed a two-parameter family of actions $w_{k, a}$ of the Lie algebra $\mathfrak{sl}(2,\mathbb{R})$  by differential-difference operators on $\mathbb{R}^n/\{0\},$ where the parameter $k$ is the multiplicity function in the Dunkl theory and $a>0$ arises from the interpolation of the Weil representation and the minimal unitary representation of the conformal group. The action $w_{k, a}$ lifts to a unitary representation of the universal covering of $SL(2, \mathbb{R})$, and is extended to a holomorphic semigroup $\Omega_{k, a}$. One boundary value of the semigroup  $\Omega_{k,a}$ provides the so-called $(k,a)$-generalized Fourier transform $\mathcal{F}_{k,a}$ \cite{orsted}. Under various values of $k$ and $a$ we can construct the well-known transforms like Fourier transform $( k=0, a=2)$, Dunkl transform $(k\neq 0, a=2)$ \cite{Dunkl-91}, Hankel transform  $(k=0, a=1)$ \cite{Sneddon}, and Dunkl-Hankel transform $(k\neq 0, a=1)$ \cite{S.B}. The Dunkl- Hankel transform is also known as the $(k,1)$-generalized Fourier transform. The fundamental tools for the $(k,1)$-generalized Fourier transform like translation, convolution, and maximal operators are discussed in \cite{S.B}, the Riesz transform is established in \cite{Ivanov1} and some properties and applications are discussed in \cite{Ivanov2}. Furthermore, studies on the imaginary powers of harmonic oscillators within the $(k,1)$ framework are presented in \cite{Teng}. Recently, the fractional powers of the $(k,1)$-Laplacian is studied in \cite{Ben-2025}.

\par  The motivation for this work stems from the significant role of weak-singular integral operators, such as the classical Riesz, Bessel, Flett, and parabolic potentials, and their various generalizations, in harmonic analysis and their applications. One key problem in potential theory is finding inversion formulas for these potentials, and several approaches have been developed to address this. The ``hypersingular integral technique", a powerful method for inverting potentials, has been developed and studied by Stein \cite{Stein},  Lizorkin \cite{Lizorkin-1970},  Samko \cite{Samko1993, Samko2002},  Rubin \cite{Rubin-1987, Rubin-1996}, and many others. Later, Rubin \cite{Rubin-1996, Rubin-1998} proposed a ``wavelet approach" to this problem and he along with many other mathematicians pursued a detailed study on the same. For instance, wavelet-type representations of Bessel and generalized parabolic potentials were introduced, and inversion formulas for these potentials were derived by  Aliev and  Rubin \cite{Aliev-2001, Aliev-2002},  Aliev and Eryi\u{g}it \cite{AlievEry-2002},  Sezer and  Aliev \cite{Sezer-Aliev} in their respective works. In 2005,  Aliev and  Rubin \cite{Aliev-2005} defined a new wavelet-like transform and proved explicit inversion formulas for both ordinary and generalized Riesz and Bessel potentials.  Sezer and  Aliev \cite{Sezer} further characterized Riesz potential spaces using these transforms.  In \cite{Verma-25}, the authors have investigated the bi-parametric potential operators in the Dunkl setting. Building on this prior research, we establish an inversion formula for Flett and Riesz potentials using the wavelet approach in the framework of $(k,1)$-generalized Fourier transform.  This method simplifies the complexity involved in higher-dimensional integration with weighted measures.
\par The core of this paper is organized into three sections: Bessel, Flett, and Riesz potentials. 
The Bessel potential denoted as $\mathbf{I}^\alpha_k$, is an operator on $L_k^p(\mathbb{R}^n)$ defined as
$$ \mathbf{I}^\alpha_k = (I- \|x\|\,\Delta_{k})^{-\frac{\alpha}{2}}, \,\, \alpha>0.$$
The $L^2$ theory associated with the multipliers leads to the following convolution form for the Bessel potential.
 
\begin{align*}
\mathbf{I}^\alpha_k(f)=  g^\alpha_k\ast_k f, \qquad  f\in L_k^p(\mathbb{R}^n), \quad 1\le p\le \infty,
 \end{align*}
 where  \begin{equation*}
g^\alpha_k(x) =  \frac{1}{\Gamma(\frac{\alpha}{2})}\int_0^\infty e^{-t}\,e^{\frac{-\|x\|}{t}}\,t^{-(n+2\gamma-1)}\, t^{\frac{\alpha}{2}}\, \frac{dt}{t}.
\end{equation*}
 
In addition, we introduce the Poisson semigroup (Poisson integral transform) by incorporating the Poisson kernel $P_k$, as 
\begin{equation*}
 \mathcal{P}_k^t(f)(x) =  P_k(t,\cdot) \ast_k f(x), \quad x\in \mathbb{R}^n, 
\end{equation*}
where the  Poisson kernel  $P_k(t,\xi):=C_{n,k}\, \frac{t}{\left(t^2+4\|\xi\|\right)^{2\gamma +n-\frac{1}{2}}}$   with $C_{n,k}= \frac{4^{2\gamma +n-1}\,\Gamma(2\gamma +n-\frac{1}{2})}{\sqrt{\pi}}$.

The $(k, 1)$-generalized Fourier Flett potentials $\mathcal{I}_k^\alpha$ of positive order $\alpha$ formally given by 
$$\mathcal{I}_k^\alpha = \left(I+\sqrt{-\|x\|\,\Delta_{k}}\right)^{-\alpha},$$
and the Poisson integral transform induces an integral representation say

\begin{equation*}
 \mathcal{I}_k^\alpha(f)(x) = \frac{1}{\Gamma(\alpha)} \int_0^\infty t^{\alpha-1}\,e^{-t}\,\mathcal{P}_k^t(f)(x)\,dt. 
\end{equation*}
We derive the key properties of this potential in Theorem \ref{T:4.3} and explore a wavelet-like transform that leads to an inverse formula for the Flett potential in Theorem \ref{inversion-Flett}.
\\

\noindent
We also develop a $\beta$-semigroup $ \mathcal{B}_k^{(\beta,t)}$, with the aid of the kernel function $  W_k^{(\beta,t)}(\xi)=\mathcal{F}_k\left( e^{-t\|\cdot\|^{\beta/2}}\right)(\xi) $ 
by utilizing the $(k,1)$-generalized convolution product. 
\begin{align*}
     \mathcal{B}_k^{(\beta,t)}(f)(x) =  W_k^{(\beta, t)} \ast_k f(x).
\end{align*} 
This allows us to represent the Riesz potential  $\mathbf{R}_k^{\alpha}$ 
  in a unified form as
\begin{align*}
    \mathbf{R}_k^{\alpha}(f)(x) = \frac{c_k^{-1}}{\Gamma(2\alpha/\beta)}\int_0^{\infty} t^{\frac{2\alpha}{\beta}-1}\mathcal{B}_k^{(\beta,t)}(f)(y)dt.
\end{align*}
Moreover, $ \mathcal{B}_k^{(\beta,t)}$ plays a key role in deriving the inverse of the Riesz potential. We employ the same wavelet method used for the Flett potential to construct the inverse of $\mathbf{R}_k^{\alpha}(f)$   (see Theorem \ref{Riesz-inversion}).

Finally, we extend our study to bi-parametric potentials using the  $\beta$-semigroup. For $\alpha,\beta>0$, we define the bi-parametric potentials $\mathfrak{J}_k^{(\alpha,\beta)}$ of order $\frac{\alpha}{\beta}$ by
 \begin{align*}
      \mathfrak{J}_k^{(\alpha,\beta)}= \left( I+ (-\|x\|\Delta_k)^{\frac{\beta}{2}}\right)^{-\alpha/\beta}.
 \end{align*}
The operator  $ \mathfrak{J}_k^{(\alpha,\beta)}$ can also be expressed in an integral form using the $\beta$-semigroup, given by
\begin{equation*}
 \mathfrak{J}_k^{(\alpha,\beta)}(f)(x) = \frac{1}{\Gamma(\alpha/\beta)} \int_0^\infty t^{\frac{\alpha}{\beta}-1}\,e^{-t}\,\mathcal{B}_k^{(\beta,t)}(f)(x)\,dt.  
\end{equation*}

Particularly, when $\beta = 1$ and $\beta = 2$, the bi-parametric potentials reduce to the Flett potential $(\mathfrak{J}_k^{(\alpha,1)})$ and the Bessel potential $(\mathfrak{J}_k^{(\alpha,2)})$, respectively. Furthermore, we introduce bi-parametric potential spaces, defined as the image of $L_k^p(\mathbb{R}^n)$ under the operator $\mathfrak{J}_k^{(\alpha,\beta)}$.

\par The structure of the paper is as follows: We begin by introducing fundamental tools for defining the 
$(k,1)$-generalized Fourier transform along with a discussion of translation and convolution operators. Further, we recall the heat semigroup and its basic properties. In Section \ref{S:3}, we establish the Bessel potential and then define the Flett potential associated with the Poisson integral. Sequentially, we discuss wavelet-like transform, which leads to the study of the inversion of the Flett potential.  We introduce the  $\beta$-semigroup incorporated with the Riesz potential and Poisson kernel,  which facilitates the inversion of the Riesz potential and further examine the Riesz potential space in Section \ref{S:4}. Finally, we unify the concepts from  Bessel and Flett potentials, along with the 
$\beta$-semigroup, to present the bi-parametric potential family and their corresponding spaces in Section \ref{S:5}.

\begin{section}{Preliminaries} \label{S:2}
The $(k,1)$-generalized Fourier transform $\mathcal{F}_k,$ stems from the elementary concepts of Dunkl theory. This led us to give a brief overview of Dunkl theory and review the $(k,1)$-generalized  Fourier transform. Finally, we conclude the section with a discussion on the properties of the heat semigroup.

\subsection{Background on Dunkl theory} 
Dunkl theory generalizes the Fourier transform by incorporating the differential-difference operator associated with the reflection group and root system \cite{D1}. We refer the readers to \cite{D2, D3, D4} for more details. \\
Throughout this paper, we consider the Euclidean space $\mathbb{R}^n$, equipped with
 the inner product $\langle x, y \rangle = \sum_{j=1}^n x(j)y(j)$.
The root system $\mathcal{R}$ is a finite collection of non-zero vectors in $\mathbb{R}^n$ that satisfies the following properties:
 \begin{enumerate}[$(i)$]
      \item $ \mathcal{R} \cap \mathbb{R}u = \{\pm u\} $
     \item  $ \sigma _u(\mathcal{R}) = \mathcal{R}$ for each $u \in \mathcal{R},$ 
 \end{enumerate} where $\sigma_u$ is the reflection map with respect to the hyperplane orthogonal to $u$.
The root system $\mathcal{R}$ can be decomposed into positive and negative subsystems, $\mathcal{R}_+$ and $\mathcal{R}_-$ respectively, using any hyperplane that passes through the origin. The group $W$, generated by the reflections $\{\sigma_u \,|\, u \in \mathcal{R}\}$, is known as the finite reflection group. The multiplicity function $k$ is a $W$-invariant complex function defined on the root system $\mathcal{R}$. We primarily focus only on positive-valued multiplicity functions in our analyses. \\
Let $\{e_1, \cdots e_n \}$ be the standard orthonormal basis of $\mathbb{R}^n$, and $1\leq j\leq n$ . The differential-difference operator (Dunkl operator) \cite{D1} reads
\begin{eqnarray*}
    \mathcal{T}_jf(x) = \partial_jf(x)+ \sum_{u\in \mathcal{R}_+}k(u) \langle u, e_j\rangle \frac{f(x)-f(\sigma_u(x))}{\langle u, x\rangle}, \quad f\in \mathcal{C}^{\prime}(\mathbb{R}^n).
\end{eqnarray*}  $\mathcal{T}_j$ represents a perturbation of the standard partial derivatives by reflection components. The standard partial derivatives can be retrieved with $k=0$.
The counterpart to the Euclidean Laplacian is the Dunkl Laplacian defined by $\Delta_k= \sum_{j=1}^n \mathcal{T}_j^2.$  The operator  $\Delta_k$ also admits the following representation with the aid of the usual Laplacian  $\Delta$, and gradient $\nabla$.
\begin{eqnarray*}
    \Delta_kf(x) = \Delta f(x)+ 2\sum _{u \in \mathcal{R}_+} k(u) \left( 
\frac{\langle \nabla f(x),u\rangle }{\langle u,x\rangle} - \frac{f(x)-f(\sigma_u(x))}{\langle u,x\rangle^2}\right), \quad f\in \mathcal{C}^2(\mathbb{R}^n). 
\end{eqnarray*}

In \cite{D2}, Dunkl proved the existence of an operator called intertwining operator $V_k$ with the following properties over the space of  homogeneous polynomial, $\mathcal{P}_m(\mathbb{R}^n)$,
\begin{equation*}
    \mathcal{T}_{\xi}V_k = V_k\partial_{\xi}, \,\, V_k(1)=1, \text{ 
 and  } V_k(\mathcal{P}_m(\mathbb{R}^n)) \subset  \mathcal{P}_m(\mathbb{R}^n), \quad  \text{for all } m \in \mathbb{N}.
\end{equation*}
An integral representation 
\begin{eqnarray} \label{inter-twinning} 
    V_kf(x) = \int_{\mathbb{R}^n} f(\xi)d\mu_x^k(\xi) \quad x \in \mathbb{R}^n, 
\end{eqnarray} for the operator $V_k$ has been provided by R\"osler in \cite{M}.
The measure $d\mu_x^k$ is a unique probability measure with compact support satisfies  supp$(\mu_x^k) \subset \{ \xi \in \mathbb{R}^n: \|\xi\| \leq \|x\| \}.$ The integral representation of $V_k$ given in \eqref{inter-twinning}, 
can further be extended to the space of continuous functions.\\
For a fixed non-negative multiplicity function $k,$ we define the weight function $\upsilon_k$  on $\mathbb{R}^n$ as
\begin{equation*}
\upsilon_k(x) =\lVert x\rVert^{-1} \prod_{u\in \mathcal{R}}|\langle u,x\rangle|^{k(u)}.
\end{equation*}
The weight function $v_k$ is $W$-invariant and homogeneous of degree $2\gamma-1$, where $ \gamma= \sum_{u\in \mathcal{R}_+} k(u)$.
We define the Lebesgue space associated with the weight function $v_k$ as follows. For $1\leq p <\infty,$
\begin{equation*}
L_k^p(\mathbb{R}^n) =\left\{ f:\mathbb{R}^n \rightarrow \mathbb{C}; f\text{ is measurable and } 
\left(\int_{\mathbb{R}^n}|f(x)|^p\upsilon_k(x)dx \right)^{\frac{1}{p}}< \infty
    \right\},
\end{equation*}
where  $dx$ be the Lebesgue measure on $\mathbb{R}^n$, and for $p=\infty,$
\begin{equation*}
L_k^{\infty}(\mathbb{R}^n) = \left\{ f:\mathbb{R}^n \rightarrow \mathbb{C}; f \text{ is measurable  and } \text{ ess.sup}|f(x)|< \infty
\right\}.
\end{equation*}
A function $f:\mathbb{R}^n \longrightarrow \mathbb{C}$ is said to be a radial function if there exists a function $f_0$ on the non-negative real line such that $f(x)=f_0(\|x\|)$ for all $x \in \mathbb{R}^n.$ The collection of all radial functions on $L_{k}^p(\mathbb{R}^n)$ is denoted by $L_{k,\text{rad}}^p(\mathbb{R}^n)$.
Now, we are in a position to define the  $(k,1)$-generalized Fourier transform and discuss some properties of it.

\subsection{The (k,1)-generalized Fourier Transformation}
The $(k,1)$-generalized Fourier transform \cite{Kobayashi} is defined for $f \in L_k^1(\mathbb{R}^n)$, by the  integral
\begin{equation*}
\mathcal{F}_k(f)(x)= c_{k}\int_{\mathbb{R}^n} f(\xi)\,B_k(x,\xi)\,\upsilon_k(\xi)\,d\xi, 
\end{equation*}
 where $B_k(x,\xi)$ is the kernel of the transform and $c_k^{-1}= \int_{\mathbb{R}^n}e^{-\|\xi\|}v_k(\xi)d\xi$. The kernel is explicitly given in terms of the intertwining operator $V_k$ as follows \cite[Theorem 4.24]{Kobayashi} 
 
 \begin{equation*}
B_k(x,\xi)= \Gamma \left( \frac{n-1}{2} + \frac{\gamma}{2} \right) V_k\left[ \Tilde{J}_{\frac{n-3}{2}+\gamma}\left( \sqrt{2\lVert x\rVert \lVert \xi \rVert(1+ \langle\frac{x}{\|x\|},. \rangle)}
\right) \right] \left( \frac{\xi}{\|\xi\|}\right).
 \end{equation*} The normalized Bessel function $\Tilde{J}_v$ can be delineated as  
 \begin{align} 
     \Tilde{J}_v(\xi) =\left( \frac{\xi}{2}\right)^{-v}J_v(\xi) = \sum_{m=0}^{\infty} \frac{(-1)^m \xi ^{2m}}{2^{2m}m!\Gamma(v+m+1)}, \label{Normalized-Bseesel}
\end{align} where $v$ is the order of the Bessel function.
 The kernel $B_k$ is uniformly bounded, i.e., for all $x, y \in \mathbb{R}^n$,  $|B_k(x,y)|\leq 1$   and  in particular, $B_k(0,y)=1$. In addition, it satisfies the relation $ \|x\| \Delta_k^x B_k(x,y) =-\|y\|B_k(x,y) $. 
\\ Now, we list the fundamental properties of the $(k,1)$-generalized Fourier transform \cite{Johansen-2016, Kobayashi}:
\begin{proposition} \label{prop-Fk}
\begin{enumerate}[$(i)$]
\item  Involutory: For every $f\in  L_k^2(\mathbb{R}^n)$,  $\mathcal{F}^{-1}_k (f) = \mathcal{F}_k(f)$. \\
\item Plancherel's formula: For all $f\in  L_k^2(\mathbb{R}^n)$, we have  $$ \|\mathcal{F}_k(f)\|_{L_k^2(\mathbb{R}^n)} = \|f\|_{L_k^2(\mathbb{R}^n)}.$$ 
\item  If $f \in L_k^1(\mathbb{R}^n)\cap L_k^2(\mathbb{R}^n)$, is a radial function, then   
\begin{equation*}
\mathcal{F}_k(f)(\xi) =\mathcal{H}_{n+2\gamma-2}f_0(\|\xi\|),
\end{equation*} where $\mathcal{H}_v$ is the Hankel transform of order $v$
\begin{equation*}
\mathcal{H}_\nu f_0(s) = \int_0^\infty f_0(u)\, \Tilde{J}_\nu(2\sqrt{us})\,u^\nu\,du.
\end{equation*}
\end{enumerate}
\end{proposition}
\textbf{Translation, Convolution, and Maximal function} are essential tools for further investigations. A detailed list of properties regarding the translation, convolution, and Hardy-Littlewood maximal function associated with the $(k,1)$-generalized Fourier transform is provided in \cite{S.B}. In order to discuss the translation operator, we define the space $A_k(\mathbb{R}^n),$ where
\begin{equation*}
 A_k(\mathbb{R}^n)=\left\{ f\in L_k^1(\mathbb{R}^n):\mathcal{F}_k(f) \in L_k^1(\mathbb{R}^n) \right\}.
\end{equation*} 
The class $A_k(\mathbb{R}^n)$ is non-empty due to the rapidly decreasing nature of   
$ \mathcal{F}_k\left(\mathcal{S}(\mathbb{R}^n)\right)$ \cite{Grobachev}, and it is contained in the intersection of $L_k^1(\mathbb{R}^n)$ and $L_k^{\infty}(\mathbb{R}^n)$. Hence, it forms a subspace of $L_k^p(\mathbb{R}^n)$, for all $p$.\\
The translation $\tau_yf$, for  functions in $A_k(\mathbb{R}^n)$, is given by 
\begin{equation}\label{eq:2.1}
\tau_yf(x) = c_{k}\int_{\mathbb{R}^n} \mathcal{F}_k(f)(z)\, B_k(x,z)\,B_k(y,z)\,v_k(z)\,dz.
\end{equation}
This formula also holds for functions in $L_k^2(\mathbb{R}^n)$. 
We now summarize some well-known properties of the translation operator $\tau_y$ \cite{S.B}.
\begin{proposition}\label{p:2.1} 
\begin{enumerate}[$(i)$]
\item Let $f \in L_{k, \text{rad}}^1(\mathbb{R}^n)$ be a bounded non-negative function. Then for all $ y \in \mathbb{R}^n$ the translation operator $\tau_y f\geq 0$.
\item For $f \in L^1_{\text{rad},k}(\mathbb{R}^n),$
\begin{equation*}
\int_{\mathbb{R}^n} \tau_y f(x)\,\upsilon_k(x)\,dx = \int_{\mathbb{R}^n} f(x)\, \upsilon_k(x)\,dx.
\end{equation*}
\item For  $ f\in L_k^2(\mathbb{R}^n)$, $ \mathcal{F}_k(\tau_y f)(\xi)= B_k(y,\xi)\,\mathcal{F}_k(f)(\xi)$.
\end{enumerate}
\end{proposition}
The following theorem explains the boundedness of the translation operator.
 \begin{theorem} \label{Thrm-translation}
    The generalized translation operator $\tau_y,$ can be extended to all radial functions in $L_k^p(\mathbb{R}^n)$, for $1\leq p\leq 2.$ Moreover $\tau_y : L_{\text{rad},k}^p(\mathbb{R}^n) \rightarrow L_k^p(\mathbb{R}^n)$ is a bounded operator.
  \end{theorem}
  
\begin{definition}\cite{S.B}  Let $f,g \in L_k^2(\mathbb{R}^n)$. Then the convolution product associated with the $(k,1)$-generalized Fourier transform of $f$ and $g$ is
\begin{equation}\label{e:2.1}
f\ast_kg(x) = c_k\int_{\mathbb{R}^n} f(y)\,\tau_xg(y)\,\upsilon_k(y)\,dy.
\end{equation}
\end{definition}
The equation \eqref{e:2.1} can be refined as
\begin{equation} \label{eq:2.3}
f\ast_kg(x) =c_k \int_{\mathbb{R}^n} \mathcal{F}_k(f)(\xi)\,\mathcal{F}_k(g)(\xi)\,B_k(x,\xi)\,\upsilon_k(\xi)\,d\xi.  \end{equation}
Below we listed some elementary properties of the convolution operator.
\begin{proposition}\label{t:2.3}
\begin{enumerate}[$(i)$]
    \item  If $g \in L_{\text{rad},k}^1(\mathbb{R}^n)$ and  $f\in L_k^p(\mathbb{R}^n)$ for $1\le p \le \infty$, then $f\ast_k g\in  L_k^p(\mathbb{R}^n)$ and  we have
\begin{equation} \label{e:2.3}
\|f\ast_k g\|_{L_k^p(\mathbb{R}^n)} \le c_k \|f\|_{L_k^p(\mathbb{R}^n)}  \|g\|_{L_{k}^1(\mathbb{R}^n)}.   
\end{equation}
\item Commutativity: $f\ast_k g = g\ast_k f$.\\
\item  $ \mathcal{F}_k(f\ast_kg)(\xi) = \mathcal{F}_kf(\xi)\mathcal{F}_k(g)(\xi),$ for every $f,g \in L_k^2(\mathbb{R}^n).$
\end{enumerate}
\end{proposition}

We now recall the maximal operator $\mathcal{M}_k$ which is necessary for proving the pointwise convergence of Flett potential.
\begin{definition} \cite{S.B}
Let $f$ be a locally integrable function on $\mathbb{R}^n$. Then the maximal function 
\begin{equation*}
\mathcal{M}_kf(x):= \mathop{\underset{r>0}{\text{sup}}}\frac{1}{\upsilon_k(B(0,r))}\left|\int_{\mathbb{R}^n}f(y)\, \tau_x(\chi_{B(0,r)})(y)\, v_k(y)dy\right|, 
\end{equation*}  
where $B(0,r)$  is the open ball of radius $r$ and centered at origin and \begin{align*}
    v_k(B(0,r))= \int_{B(0,r)}v_k(\xi)\,d\xi.
\end{align*}
Using the spherical co-ordinates $\xi = r\xi^\prime,$ where $\xi^\prime \in \mathbb{S}^{n-1},$ we have 
\begin{align}
    v_k(B(0,r))= \int_o^{r} \int_{\mathbb{S}^{n-1}}   v_k(\xi^\prime) d\omega(\xi^\prime) t^{ 2\gamma+n-2}dt = \frac{S_k }{2\gamma+n-1} r^{2\gamma+n-1}, \label{wt-1}
\end{align}
where $S_k :=\int_{\mathbb{S}^{n-1}}   v_k(\xi^\prime) d\omega(\xi^\prime)$. Here the measure $d\omega$ denotes the Lebesgue surface measure on the unit sphere $\mathbb{S}^{n-1}.$
\end{definition}
\end{section}

To demonstrate Bessel potential, the heat transform must be studied.

\subsection{ Heat transform for (k,1)-generalized Fourier transform } In this part we recall some properties of heat transform. A detailed review can be found in \cite{S.B}. The function $F_k(\cdot, t): \mathbb{R}^n \rightarrow \mathbb{R}$ defined by $$ F_k(x,t) = \frac{c_k}{t^{n+2\gamma-1}}\,e^{-\frac{\|x\|}{t}},\,\, t>0$$ is the solution of the heat equation $\|x\|\, \Delta_k\,u(x,t)= \partial_t\,u(x,t)$.  Associated with the function $F_k$, we define the heat kernel $h_k$ on $\mathbb{R}^n\times\mathbb{R}^n\times (0,\infty)$ by 
$$ h_k(x,y,t):= \tau_y(F_k(\cdot,t))(x).$$  The  definition of translation leads to the following integral representation 
\begin{equation}\label{e:3.1}
   h_k(x,y,t)= c_k^2\,\int_{\mathbb{R}^n} e^{-t\|\xi\|}\,B_k(x,\xi)\,B_k(y,\xi)\, \upsilon_k(\xi)\,d\xi.
\end{equation}
It is also a solution of the heat equation and  we discuss some basic properties of $h_k(x,y,t)$ below:
\begin{proposition}\label{p:3.1}
For all $(x,y,t) \in \mathbb{R}^n  \times \mathbb{R}^n \times (0,\infty)$, we have
\begin{enumerate}[$(i)$]
\item $|h_k(x,y,t)|\le \frac{c_k}{t^{n+2\gamma -1}}\,e^{\frac{-(\|x\|^{1/2}-\|y\|^{1/2})^2}{t}}.$
\item $\mathcal{F}_k(F_k(\cdot,t))(\xi) = e^{-t\|\xi\|}.$ 
\end{enumerate}
\end{proposition}
\begin{definition}
The heat transform of a smooth measurable function $f$ on $\mathbb{R}^n$  is given by
\begin{equation*}
    H_k^t(f)(x) =   \left \{
\begin{array}{ll}
c_k \int_{\mathbb{R}^n} \tau_y\left(F_k(\cdot,t)\right)(x)\,f(y)\, \upsilon_k(y)\,dy,& \text{for}\quad t>0,  \\ \\
f(x) ,& \text{for}\quad t=0.
\end{array}
\right. 
\end{equation*}
\end{definition}
 The boundedness of the heat transform is guaranteed by the following theorem.
\begin{theorem}
    \begin{itemize}
        \item [(i)] For every $t>0$,\,\, $H_k^t$ is a continuous linear operator on $L_k^p(\mathbb{R}^n)$, with $$\|H_k^t(f)\|_{L_k^p(\mathbb{R}^n)}\le \|f\|_{L_k^p(\mathbb{R}^n)}.$$
        \item[(ii)] $(H_k^t)_{t\ge0}$ is a semi group satisfying $\|H_k^t(f)-f\|_{L_k^\infty(\mathbb{R}^n)} \rightarrow 0$ as $t \rightarrow 0$.
    \end{itemize}
\end{theorem}

\section{Bessel and Flett potentials associated to \texorpdfstring{$(k,1)$}{(k,1)}-generalized Fourier transform } \label{S:3}
The Bessel and Flett potential are essential for studying different types of singularities. We start by defining the Bessel potential using the heat kernel and then employ the Poisson integral to construct the Flett potential. Furthermore, a wavelet-like transform is introduced to establish the inversion formula for the Flett potential. 

\subsection{Bessel Potential}The \texorpdfstring{$(k,1)$}{(k,1)}-generalized Fourier Bessel potential is defined as
$$ \mathbf{I}^\alpha_k = (I- \|x\|\, \Delta_{k})^{-\frac{\alpha}{2}}, \,\, \alpha>0.$$
Further, we derive the convolution representation for  $ \mathbf{I}^\alpha_k$ which is defined by 
\begin{align}
\mathbf{I}^\alpha_k(f)=  g^\alpha_k\ast f,
 \label{Bessel-potential}
\end{align}
where $g_k^\alpha$ satisfies $\mathcal{F}_k( g^\alpha_k)(x) = \frac{1}{(1+\|x\|)^{\frac{\alpha}{2}}}$ and for any $f\in L_k^p(\mathbb{R}^n)$, $1\le p\le \infty$. Eventually, $g_k^\alpha$ can be described as:
 \begin{equation}\label{2.1}
g^\alpha_k(x) =  \frac{1}{\Gamma(\frac{\alpha}{2})}\int_0^\infty e^{-t}\,e^{\frac{-\|x\|}{t}}\,t^{-(n+2\gamma-1)}\, t^{\frac{\alpha}{2}}\, \frac{dt}{t}.
\end{equation}
 In the forthcoming proposition, we discuss the fundamental properties of the $g_k^{\alpha}.$
\begin{proposition} \label{Bessel kernal}
Let $g^\alpha_k$ be defined in \eqref{2.1}. Then  $g^\alpha_k(x)\ge0$, for all $x\in \mathbb{R}^n$ and $g_k^{\alpha}\in L_{k, \text{rad}}^1(\mathbb{R}^n)$.
\end{proposition}
\begin{proof}
Given that
\begin{equation*}
g^\alpha_k(x) =  \frac{c_k}{\Gamma(\frac{\alpha}{2})}\int_0^\infty e^{-t}\,e^{\frac{-\|x\|}{t}}\,t^{-(n+2\gamma-1)}\, t^{\frac{\alpha}{2}}\, \frac{dt}{t}.
\end{equation*} 
The positivity and radial properties of the function $g_k^{\alpha}$  directly follows from \eqref{2.1}. In view of Tonelli's theorem, we see that 
\begin{align}
\nonumber    \int_{\mathbb{R}^n}  g^\alpha_k(x)\, \upsilon_k(x)\,dx &= \frac{c_k}{\Gamma(\frac{\alpha}{2})}\int_0^\infty   e^{-t}\, t^{\frac{\alpha}{2}}\,t^{-(n+2\gamma-1)}\, \left(\int_{\mathbb{R}^n}e^{\frac{-\|x\|}{t}}\,  \upsilon_k(x)\,dx \right)\frac{dt}{t}\\
    & = \frac{1}{\Gamma(\frac{\alpha}{2})}\int_0^\infty e^{-t}\, t^{\frac{\alpha}{2}-1}\,dt = 1 \label{e3.3}
\end{align}
which completes the proof.
\end{proof}
The boundedness and semigroup property of the Bessel potential are furnished in the following theorem.
\begin{theorem} \label{t:3.2}
If $f\in L_k^p(\mathbb{R}^n)$, $1 \le p \le \infty$ and $\alpha>0$, then $ \mathbf{I}^\alpha_k$ is a bounded operator from $L_k^p(\mathbb{R}^n)$ to itself and we have
\begin{equation*}
    \| \mathbf{I}^\alpha_k (f)\|_{L_k^p(\mathbb{R}^n)} \le c_k\,\|f\|_{L_k^p(\mathbb{R}^n)}. 
\end{equation*}
Further, for $\alpha,\beta>0$ $$ \mathbf{I}^\alpha_k(\mathbf{I}^\beta_k(f)) = \mathbf{I}^{\alpha+\beta}_k(f).$$
\end{theorem}
\begin{proof}
Let $f\in L_k^p(\mathbb{R}^n)$ for $1\le p \le \infty.$ In conjuction with 
\eqref{Bessel-potential}, \eqref{e3.3} and Proposition \ref{t:2.3} , we have
\begin{eqnarray*}
 \| \mathbf{I}^\alpha_k (f)\|_{L_k^p(\mathbb{R}^n)} &=&\| g^\alpha_k\ast_k f\|_{{L_k^p(\mathbb{R}^n)}}\\
 &\leq & c_k\|f\|_{L_k^p(\mathbb{R}^n)}\, \|g^\alpha_k\|_{L_k^1(\mathbb{R}^n)}\\
 &=&c_k\, \|f\|_{L_k^p(\mathbb{R}^n)}.
 \end{eqnarray*}
Thus, we can see that the $(k,1)$-generalized Fourier Bessel potentials $ \mathbf{I}^\alpha_k(f)$ are the bounded operators. Now, we are in a position to prove the semigroup property. Let $f\in \mathcal{S}(\mathbb{R}^n)$,
\begin{eqnarray*}
\mathbf{I}^\alpha_k(\mathbf{I}^\beta_k(f))(x)  &=& g^\alpha_k\ast\mathbf{I}^\beta_k(f)(x)\\&=& c_k\int_{\mathbb{R}^n} g_k^\alpha(y)\, \tau_x(g_k^\beta \ast f)(y)\,\upsilon_k(y)\,dy.
\end{eqnarray*} 
Invoking  \eqref{eq:2.1} and  the third assertion of Proposition \ref{t:2.3} in the above equation
\begin{eqnarray*}
\mathbf{I}^\alpha_k(\mathbf{I}^\beta_k(f))(x)&=& c_k^2 \int_{\mathbb{R}^n} \int_{\mathbb{R}^n}g^\alpha_k (y)\,B_k(x,\xi)\,  B_k(y,\xi)\, \mathcal{F}_k(g_k^\beta )(\xi)\, \mathcal{F}_k(f)(\xi)\, \upsilon_k(\xi)\, \upsilon_k(y)\,d\xi\,dy.
\end{eqnarray*}
 Fubini's theorem allows us to change the order of integration, and then we have
\begin{equation*}
\mathbf{I}^\alpha_k(\mathbf{I}^\beta_k(f))(x)  = c_k\int_{\mathbb{R}^n} \mathcal{F}_k(g_k^\alpha)(\xi)\, \mathcal{F}_k(g_k^\beta)(\xi)\,\mathcal{F}_k(f)(\xi)\,B_k(x,\xi)\,\upsilon_k(\xi)\,d\xi.
\end{equation*}
 On account of that $\mathcal{F}_k( g^\alpha_k)(x) = \frac{1}{(1+\|x\|)^{\frac{\alpha}{2}}}$, we obtain
\begin{eqnarray*}
&=& c_k\int_{\mathbb{R}^n} \frac{1}{(1+\|x\|)^{\frac{\alpha+\beta}{2}}}\,\mathcal{F}_k(f)(\xi)\,B_k(x,\xi)\,\upsilon_k(\xi)\,d\xi\\
&=& c_k \int_{\mathbb{R}^n} \mathcal{F}_k(g^{\alpha+\beta}_k)(\xi)\,\mathcal{F}_k(f)(\xi)\, B_k(x,\xi)\,v_k(\xi)\,d\xi.
\end{eqnarray*}
One can  easily check that $g^{\alpha+\beta}_k$ belongs to $L_k^2(\mathbb{R}^n)$. Because of \eqref{eq:2.3}, the above equation can be revamped regarding convolution. That is,
\begin{eqnarray*}
\mathbf{I}^\alpha_k(\mathbf{I}^\beta_k(f))(x)&=& ( g^{\alpha+\beta}_k\ast f)(x)\\
&=& \mathbf{I}^{\alpha+\beta}_k(f)(x).
\end{eqnarray*}
This completes the proof.
\end{proof}
The convolution definition of the Bessel potential, along with the explicit formula for $g_k^{\alpha}$, leads us to express the Bessel potential in terms of the heat transform. Also, we establish the relationship between the heat transform and the Bessel potential in the following theorem:
\begin{theorem}
If   $f\in L_k^p(\mathbb{R}^n)$, $1 \le p < \infty$, $f \in C_0(\mathbb{R}^n)$ for $p=\infty$ and $\alpha>0$, then
\begin{enumerate}[$(i)$]
\item  $(k,1)$- generalized Fourier Bessel potential $ \mathbf{I}^\alpha_k(f) $ of order $\alpha$ 
is expressed as
\begin{equation}\label{e:3.3}
 \mathbf{I}^\alpha_k(f) (x) =   \frac{1}{\Gamma(\frac{\alpha}{2})}\int_0^\infty t^{\frac{\alpha}{2}-1}\,e^{-t}\,H_k^t(f)(x)\,dt.
\end{equation}
\item For $s,t>0$, the heat transform of $ \mathbf{I}^\alpha_k(f)$ is the function $H_k^s( \mathbf{I}^\alpha_k)$ given by 
\begin{equation*}
    H_k^s( \mathbf{I}^\alpha_k(f))(x) = \frac{1}{\Gamma(\frac{\alpha}{2})}\int_0^\infty t^{\frac{\alpha}{2}-1}\,e^{-t}\, H^{s+t}_k(f)(x)\,dt.
\end{equation*}
Moreover, $H_k^s( \mathbf{I}^\alpha_k(f))(x)=\mathbf{I}^\alpha_k(H_k^s(f))(x)$.
\end{enumerate}
\end{theorem}
\begin{proof}$(i)$ 
Let $f\in L_k^p(\mathbb{R}^n)$ for $1\le p \le \infty$.  We consider the left-hand side of equation \eqref{e:3.3} 
\begin{eqnarray*}
 \mathbf{I}^\alpha_k(f)(x) &=& g_k^\alpha \ast f(x) \\ 
 &=& \frac{c_k}{\Gamma(\frac{\alpha}{2})} \int_{\mathbb{R}^n}\int_0^\infty e^{-t}\,e^{-\frac{\|y\|}{t}}\,t^{\frac{\alpha}{2}-1}\,t^{-(n+2\gamma-1)}\, \tau_xf(y)\,v_k(y)\,dt\,dy.
 \end{eqnarray*}
 In regard to Fubini's theorem, Proposition \ref{p:3.1}(ii) and  \eqref{eq:2.1}, we obtain
\begin{equation}\label{e:3.5}
 \mathbf{I}^\alpha_k(f)(x)= \frac{c_k}{\Gamma(\frac{\alpha}{2})} \int_0^\infty \int_{\mathbb{R}^n}e^{-t}\,t^{\frac{\alpha}{2}-1}\, e^{-t\|\xi\|}\,\mathcal{F}_k(f)(\xi)\,B_k(x,\xi)\,\upsilon_k(\xi)\,d\xi\,dt.
\end{equation}
Applying the definition heat transform, the right-hand side of \eqref{e:3.3} becomes
\begin{eqnarray*}
&&\frac{1}{\Gamma(\frac{\alpha}{2})}\int_0^\infty t^{\frac{\alpha}{2}-1}\,e^{-t}\,H_k^t(f)(x)\,dt\\
&=& \frac{c_k}{\Gamma(\frac{\alpha}{2})}\int_0^\infty\int_{\mathbb{R}^n}t^{\frac{\alpha}{2}-1}\,e^{-t}\,\tau_y(F_k(\cdot,t))(x)\,f(y)\,\upsilon_k(y)\,dy\,dt.
\end{eqnarray*}
The result follows from \eqref{e:3.1} and \eqref{e:3.5} applied to the above equality.\\

\noindent $(ii)$  By applying the definition of the heat transform and the semigroup property, we derive the expression for $H_k^s( \mathbf{I}^\alpha_k(f))$. Let us consider
\begin{equation*}
\mathbf{I}^\alpha_k(H_k^s(f))(x)= \frac{1}{\Gamma(\frac{\alpha}{2})}\int_0^\infty t^{\frac{\alpha}{2}-1}\,e^{-t}\,H_k^t(H_k^s(f))(x)\,dt.
\end{equation*}
It is enough to prove that $H_k^s(H_k^t(f))(x)= H^{s+t}_k(f)(x)$. By using Proposition \ref{p:3.1} and Proposition \ref{t:2.3}, we deduce that
\begin{eqnarray*}
 H_k^s(H_k^t(f))(x)&= &H_k^s(f\ast F_k(\cdot,t))(x)\\
 &=& (f\ast F_k(\cdot,t))\ast F_k(\cdot,s)(x)\\
 &=& c_k\int_{\mathbb{R}^n}\mathcal{F}_k(f\ast F_k(\cdot,t))(\xi)\,\mathcal{F}_k(F_k(\cdot,s))(\xi)\,B_k(x,\xi)\,\upsilon_k(\xi)\,d\xi\\
 &=& c_k\int_{\mathbb{R}^n}\mathcal{F}_k(f)(\xi)\,e^{-(t+s)\|\xi\|}\,B_k(x,\xi)\,\upsilon_k(\xi)\,d\xi\\
 &=&c_k\int_{\mathbb{R}^n}\mathcal{F}_k(f)(\xi)\,\mathcal{F}_k(F_k(\cdot,s+t))(\xi)\,B_k(x,\xi)\, \upsilon_k(\xi)\, d\xi\\
 &=& f\ast F_k(\cdot,s+t)(x).
\end{eqnarray*}
This gives the required result. Moreover, it is clear that $$ H_k^s( \mathbf{I}^\alpha_k(f))(x)=\mathbf{I}^\alpha_k(H_k^s(f))(x).$$
\end{proof}
\subsection{The Poisson integral transform}
This subsection serves as a prerequisite for defining the Flett potential. Recently, Ben Sa\"id studied the Poisson integral operator for the $(k,a)$-generalized Fourier transform on the real line \cite{Negzaoui}.  Here, we focus on the generalized Poisson integral operator $\mathcal{P}_k^t$ for $t>0$, applied to functions on $\mathbb{R}^n$  within the framework of the $(k,1)$-generalized Fourier transform.

 \begin{definition} \label{D:4.1}
Let $1\le p\le \infty$ and $f\in L^p_k(\mathbb{R}^n)$. Then the generalized Poisson integral operator $\mathcal{P}_k^t$ defined on $L^p_k(\mathbb{R}^n)$, is given by
\begin{equation*}
 \mathcal{P}_k^t(f)(x) =  P_k(t,\cdot){\ast} f(x), \quad x\in \mathbb{R}^n, 
\end{equation*}
where the  Poisson kernel is defined as
$$P_k(t,\xi):=C_{n,k}\, \frac{t}{\left(t^2+4\|\xi\|\right)^{2\gamma +n-\frac{1}{2}}}$$
with $C_{n,k}= \frac{4^{2\gamma +n-1}\,\Gamma(2\gamma +n-\frac{1}{2})}{\sqrt{\pi}}$.
\end{definition} 
It is immediately follows that $P_k(t,\cdot)$ is a positive radial function. In conjunction, we discuss some fundamental properties of the Poisson kernel $P_k(t,\cdot)$ which can be readily verified:
\begin{proposition} \label{p:4.1}
\begin{enumerate}[$(i)$]
\item  $\mathcal{F}_k(P_k(t,.))(\xi) = e^{-t\,\|\xi\|^{\frac{1}{2}}}$\quad for all $t>0.$  \\
\item  $\int_{\mathbb{R}^n} P_k(t,x)\,\upsilon_{k}(x)dx=1$. 

\end{enumerate}
\end{proposition}
\begin{proof}
$(i)$. We need to prove that $\mathcal{F}^{-1}_k(e^{-t\|\cdot\|^\frac{1}{2}})(\xi)=P_k(t,\xi)$.
We know that $\mathcal{F}_k(f) = \mathcal{F}^{-1}_k(f)$ and proceed with the proof by using the connection between the Hankel transform and $\mathcal{F}_k$, we have 
\begin{eqnarray*}
\mathcal{F}_k(e^{-t\|\cdot\|^\frac{1}{2}})(\xi) &=&  \mathcal{H}_{n+2\gamma-2}(e^{-t\|\cdot\|^\frac{1}{2}})(\xi)\\
&=& \int_0^ \infty e^{-t\sqrt{u}}\, \left(\sqrt{u\xi}\right)^{n+2\gamma-2}\,J_{n+2\gamma-2}\left(2\sqrt{u\xi}\right)\,u^{n+2\gamma-2}\,du\\
&=& C_{n,k}\, \frac{t}{\left(t^2+4\|\xi\|\right)^{2\gamma +n-\frac{1}{2}}}.
\end{eqnarray*}
The proof of $(ii)$  follows directly from the property of $B(0,y)=1.$
\end{proof}

We establish an upper bound for the family of convolutions formed by a dilated function  $g_{\epsilon}(x)=\epsilon^{-(n+2\gamma-1)}g(\frac{x}{\epsilon})$, using the maximal function as a tool. A similar result can be found in the Dunkl setting also \cite{Thangavelu}. 
 
\begin{theorem} \label{t:4.9}
    Let $f \in L_k^p(\mathbb{R}^n)$, $1\le p < \infty$ and $g \in A_k(\mathbb{R}^n)$ be a real valued radial function which satisfies $|g(y)| \le c\left( 1+4\|y\|\right)^{-(n+2\gamma-\frac{1}{2})}$. Then 
    \begin{align*}
        \sup_{\epsilon >0} |f \ast_k g_{\epsilon} (y)| \le c\,\mathcal{M}_kf(y). 
    \end{align*}
\end{theorem} 
\begin{proof}
    We begin the proof by assuming $f,g \ge 0.$ Representing the function $g_{\epsilon}$ in terms of characteristic function as
    \begin{align*}
        g_{\epsilon} (\xi) &= g_{\epsilon}(\xi)\chi_{B[0,\epsilon]} (\xi) + \sum_{j=1} ^{\infty} g_{\epsilon}(\xi) \chi_{\left(B[0,\epsilon 2^{j}]\setminus B[0,\epsilon 2^{j-1}]\right)}(\xi),
    \end{align*} where $B[0,r]$ is the closed ball of radius $r$ centered at the origin. 
    Consider the partial sum, 
    \begin{align*}
        \sum_{j=1} ^{m} g_{\epsilon}(\xi) &  \chi_{\left(B[0,\epsilon  2^{j}]\setminus B[0,\epsilon 2^{j-1}]\right)}(\xi) \\
        & \le c \epsilon^{-(n+2\gamma-1)} \sum_{j=1} ^{m} 
        \left( 1+ 2^{j+1}\right)^{-(n+2\gamma-\frac{1}{2})}
         \chi_{\left(B[0,\epsilon 2^{j}]\setminus B[0,\epsilon 2^{j-1}]\right)} (\xi) 
    \end{align*}
     and on integrating we obtain
 \begin{align}
    & \int_{\mathbb{R}^n} f(\xi)   \tau_y \left(
     \sum_{j=1} ^{m} g_{\epsilon}(\xi) 
     \chi_{\left(B[0,\epsilon  2^{j}]\setminus B[0,\epsilon 2^{j-1}]\right)}\right)(\xi) v_k(\xi)d\xi \notag
     \\ & \le c \epsilon^{-(n+2\gamma-1)} \sum_{j=1} ^{m}  \left( 1+ 2^{j+1}\right)^{-(n+2\gamma-\frac{1}{2})} \int_{\mathbb{R}^n} f(\xi)
       \tau_y \left( \chi_{\left(B[0,\epsilon  2^{j}]\setminus B[0,\epsilon 2^{j-1}]\right)}\right)(\xi) v_k(\xi)d\xi. \label{thrm3.6-1}
\end{align}
Using the definition of maximal function and \eqref{wt-1}, for all $j\in \mathbb{N}$, we have 
\begin{align}
    \int_{\mathbb{R}^n} f(\xi)
       \tau_y \left( \chi_{\left(B[0,\epsilon  2^{j}]\setminus B[0,\epsilon 2^{j-1}]\right)}\right)(\xi) v_k(\xi)d\xi \leq S_k\frac{(\epsilon 2^{j})^{n+2\gamma-1} }{n+2\gamma-1}  \mathcal{M}_k f(y). \label{thrm3.6-2}
\end{align}
Invoking the estimate \eqref{thrm3.6-2} in \eqref{thrm3.6-1} we get
\begin{align*}
    & c \frac{S_k}{n+2\gamma-1} 
    \sum_{j=1} ^{m}   \frac{( 2^{j})^{n+2\gamma-1} } {\left(1+ 2^{j+1}\right)^{(n+2\gamma-\frac{1}{2})}}
   \mathcal{M}_kf(y)   \le \Tilde{c}\mathcal{M}_kf(y).
 \end{align*}  Due to the convergence of the above series we can find a constant $\Tilde{c},$ independent of the value $m$.
We can obtain a similar bound by including the $g_{\epsilon}(\xi)\chi_{B[0,\epsilon]} (\xi)$ as $j=0$. Thus,
\begin{align}
    \int_{\mathbb{R}^n} f(\xi)   \tau_y \left(
     \sum_{j=0} ^{m} g_{\epsilon}(\xi) 
     \chi_{\left(B[0,\epsilon  2^{j}]\setminus B[0,\epsilon 2^{j-1}]\right)}\right)(\xi) v_k(\xi)d\xi \leq C \mathcal{M}_kf(y). \label{Thrm3.6-3}
\end{align} The boundedness of the translation operator $\tau_y$ (Theorem \ref{Thrm-translation}) will imply the desired result by passing the limit $m \rightarrow \infty$ over the inequality \eqref{Thrm3.6-3}.
\begin{align*}
   f \ast_kg_\epsilon(y) \leq C\mathcal{M}_kf(y) \text{ for all } \epsilon >0. 
\end{align*} 
\end{proof}  
We conclude this subsection by discussing some behavior of the Poisson semigroup with respect to the variables $x$ and $t$, along with its boundedness and semigroup properties, which will be further examined in the forthcoming theorem.
\begin{theorem} \label{T:4.3}
For every $t>0$, we have the following results:  
\begin{enumerate}[$(i)$]
\item [(i)] The family $\{\mathcal{P}_k^t\}_{t>0}$ is uniformly bounded in $L^p_k(\mathbb{R}^n)$. For $1\le p < \infty$, $f \in L^p_k(\mathbb{R}^n)$ and $f \in C_0(\mathbb{R}^n)$ for $p=\infty$ we have
\begin{equation*}
\|\mathcal{P}_k^t(f)\|_{L^p_k(\mathbb{R}^n)} \le c_k
\|f\|_{L^p_k(\mathbb{R}^n)}.
    \end{equation*}
\item[(ii)] Let $f\in L_k^p(\mathbb{R}^n)$ with $1\le p <\infty$. Then 
$$\mathop{\underset{x \in \mathbb{R}^n}{\text{sup}}}|\mathcal{P}_k^t(f)(x)| \le  c_k   t^{-\frac{2}{p}(n+2\gamma-1)} \|f\|_{L_k^p(\mathbb{R}^n)}.$$

\item[(iii)] If $f\in L_k^p(\mathbb{R}^n)$ and  $1\le p<\infty$, then $$\underset{t\rightarrow 0}{\lim} \,\mathcal{P}_k^t(f)(x) = f(x)$$ converges point wise a.e. For $f\in C_0(\mathbb{R}^n)$, the convergence is uniform on $\mathbb{R}^n$.
\item [(iv)] Semigroup property: $\mathcal{P}_k^t\circ\mathcal{P}_k^s\,(f)= \mathcal{P}_k^{t+s}(f)$,\quad $t,s>0.$
\item[(v)] $$\mathop{\underset{t>0}{\text{sup}}}|\mathcal{P}_k^t(f)(x)| \le c\, \mathcal{M}_kf.$$
\end{enumerate}
\end{theorem}
\begin{proof} 
\begin{enumerate}[$(i)$]
    \item The proof follows from the definition of the Poisson integral transform and  \eqref{e:2.3}.
    \item 
     Now we consider the point-wise bound of  $\mathcal{P}_k^{t}(f)(x)$ 
    \begin{align*}
       \big| \mathcal{P}_k^{t}(f)(x) \big| \le c_k \int_{\mathbb{R}^n} \big| \tau_x P_k{(t, \cdot)}(y)f(y)\big|v_k(y)dy. 
    \end{align*}
    In sight of H\"older's inequality, we get 
    \begin{align*}
        \int_{\mathbb{R}^n}& \big| \tau_x P_k{(t, \cdot)}(y)f(y)\big|v_k(y)dy
        \\ & \le \left( \int_{\mathbb{R}^n} |f(y)|^p|\tau_xP_k{(t, \cdot)}(y)| v_k(y)dy
        \right)^{1/p} \left( \int_{\mathbb{R}^n} |\tau_xP_k{(t, \cdot)}(y)|v_k(y)dy \right)^{1/q}. 
    \end{align*} Invoking the properties of the translation and the Proposition \ref{p:4.1}, the above inequality reduces to
    \begin{align*}
          & \int_{\mathbb{R}^n} \big| \tau_x P_k{(t, \cdot)}(y)f(y)\big|v_k(y)dy \\
          & \le  c_k \left( \int_{\mathbb{R}^n} |f(y)|^pv_k(y)dy \, c_k\int_{\mathbb{R}^n} e^{-t\|\xi\|^{1/2}} v_k(\xi)d\xi
          \right)^{1/p} \\
         & \le c_k \|f\|_{L_k^p(\mathbb{R}^n)}  t^{-\frac{2}{p}(n+2\gamma-1)}.
    \end{align*} 
This completes the proof of $(ii)$.\\

\item  We begin the proof by considering the case of $p=\infty$. Let $f\in \mathcal{S}(\mathbb{R}^n)$. Then we have
\begin{align*}
    \left| \mathcal{P}_k^t(f)(x)-f(x) \right|  \leq \int_{\mathbb{R}^n} | e^{-t\|\xi\|^{\frac{1}{2}}} -1 | \left|\mathcal{F}_k(f)(\xi)\right| v_k(\xi)d\xi .
\end{align*} Using the fact that $\mathcal{F}_k(\mathcal{S}(\mathbb{R}^n)$ is rapidly decreasing \cite{Grobachev}, along with the Dominated convergence theorem, we conclude that
\begin{align} \label{Poisson-L}
    \| \mathcal{P}_k^t(f)-f\|_{L_k^{\infty}(\mathbb{R}^n)} \rightarrow 0 \text{   as}\quad t \rightarrow 0.
\end{align}
Now, consider the case of $1\le p < \infty$.  Let $f\in C_c(\mathbb{R}^n)$. There exists an $r>0$ such that $ \text{supp}(f) \subset B(0,r)$, where $B(0,r)$ is the ball of radius $r$ centered at the origin. We notice the following 
\begin{align*}
   & \int_{\mathbb{R}^n\setminus B(0,r)}  \tau_{(t^2x)}  P_k(t,\xi) v_k(\xi)d\xi\\
    & = c_k \int_{\mathbb{R}^n \setminus B(0,r)} \int_{\mathbb{R}^n}  e^{-t\|y\|^{\frac{1}{2}}} B_k(t^2x, y)B_k(\xi,y)v_k(y)dy\, v_k(\xi)d\xi\\
    & = t^{-2(n+2\gamma-1)} c_k \int_{\mathbb{R}^n \setminus B(0,r)} \int_{\mathbb{R}^n} e^{-\|y\|^{\frac{1}{2}}} B_k(x, y)B_k(t^{-2}\xi,y)v_k(y)dy\, v_k(\xi)d\xi\\
    & = c_k \int_{\mathbb{R}^n \setminus B(0,\frac{r}{t^2})} \int_{\mathbb{R}^n} e^{-\|y\|^{\frac{1}{2}}} B_k(x, y)B_k(\xi,y)v_k(y)dy\, v_k(\xi)d\xi\\
    & =  \int_{\mathbb{R}^n\setminus B(0,\frac{r}{t^2})}  \tau_{x}  P_k(1,\xi) v_k(\xi)d\xi.
\end{align*} Since the Poisson kernel is integrable, the dominated convergence theorem leads to  the conclusion that 
\begin{align} \label{Poisson-1}
    \int_{\mathbb{R}^n\setminus B(0,r)}  \tau_{(t^2x)}  P_k(t,\xi) v_k(\xi)d\xi \rightarrow 0 \quad \text{    as $t\rightarrow 0$.}
\end{align}  
Consider the following 
\begin{align*}
    \| \mathcal{P}_k^t(f) -f\|_{L_k^p(\mathbb{R}^n)} & \le  \| \mathcal{P}_k^t(f) -f\|_{L_k^p(B(0,r))}+  \| \mathcal{P}_k^t(f) -f\|_{L_k^p(\mathbb{R}^n\setminus B(0,r))}\\
    & \le  \| \mathcal{P}_k^t(f) -f\|_{L_k^{\infty}(B(0,r))}\left( v_k(B(0,r)) \right)^{\frac{1}{p}} + \| \mathcal{P}_k^t(f) \|_{L_k^p(\mathbb{R}^n\setminus B(0,r))}.
\end{align*} 
Therefore $\| \mathcal{P}_k^t(f)-f\|_{L_k^p(\mathbb{R}^n)} \longrightarrow 0,$ as $t \longrightarrow 0$ from \eqref{Poisson-L} and \eqref{Poisson-1}.\\
\item  This result is an immediate consequence of  Definition \ref{D:4.1}, Proposition \ref{t:2.3} $(iii)$, and Proposition \ref{p:4.1} $(i)$.

\item 
It is straightforward from Theorem \ref{t:4.9}. Hence the proof is complete.
\end{enumerate}
\end{proof}
\subsection{Flett potentials }
 We introduce the $(k, 1)$-generalized Fourier Flett potentials associated with the differential operator $\|x\|\,\Delta_{k}$. The $(k,1)$-generalized Flett potentials are defined analogously to the classical Fourier Flett potentials \cite{Negzaoui}. To obtain an inverse formula for $(k, 1)$-generalized Fourier Flett potentials, a wavelet-like transform needs to be defined for the $(k, 1)$-generalized Fourier transform. 
\begin{definition}
The $(k, 1)$-generalized Fourier Flett potentials $\mathcal{I}_k^\alpha$ of positive order $\alpha$ can be defined by
$$\mathcal{I}_k^\alpha = (I+(-\|x\|\,\Delta_{k})^{\frac{1}{2}})^{-\alpha}, \quad \alpha>0.$$
To be more precise, we define $\mathcal{I}_k^\alpha$ as a convolution operator $\mathcal{I}_k^\alpha (f)=f\ast F^\alpha_k$, where  $F^\alpha_k(x)= \mathcal{F}^{-1}_{k}\left[(1+\|\xi\|^\frac{1}{2})^{-\alpha}\right](x)$.
\end{definition}
In another way, we can  express Flett's potential as
\begin{eqnarray} \label{e:5 .1}
\mathcal{F}_k(\mathcal{I}_k^\alpha (f))(\xi) = \mathcal{F}_k(f)(\xi) (1+\|\xi\|^\frac{1}{2})^{-\alpha}
\end{eqnarray}

Now, we give the integral representation for $(k, 1)$-generalized Fourier Flett potentials using the Poisson integral.
\begin{definition}
Let $f\in L_k^p(\mathbb{R}),\,\,1\le p \le \infty$   and for any $\alpha>0$. We define the $(k, 1)$- generalized Fourier Flett potentials $\mathcal{I}_k^\alpha f$  by
\begin{equation*}
 \mathcal{I}_k^\alpha(f)(x) = \frac{1}{\Gamma(\alpha)} \int_0^\infty t^{\alpha-1}\,e^{-t}\,\mathcal{P}_k^t(f)(x)\,dt. 
\end{equation*}
\end{definition} 
The following theorem provides some basic properties of  $(k, 1)$- generalized Fourier Flett potentials.
\begin{theorem} \label{t:3.10}
Let $f\in L_k^p(\mathbb{R})$, $1\le p <\infty$  and $f\in C_0(\mathbb{R})$  for $p=\infty$. Then 
\begin{enumerate}[$(i)$]
\item  $\mathcal{I}_k^\alpha$ is well defined on $L_k^p(\mathbb{R}^n)$ and it is a bounded operator,
\begin{equation*}
\|\mathcal{I}_k^\alpha (f)\|_{L_k^p(\mathbb{R}^n)} \le c_k\|f\|_{L_k^p(\mathbb{R}^n)}.
        \end{equation*}
\item The family $(\mathcal{I}_k^\alpha )_{\alpha\ge0}$ satisfies the semigroup property $\mathcal{I}_k^\alpha\circ\mathcal{I}_k^\beta = \mathcal{I}_k^{\alpha+\beta}$.
\end{enumerate}
\end{theorem}
\begin{proof}
$(i)$ We prove the boundedness of the Flett potentials by using Minkowski's integral inequality and Theorem \ref{T:4.3}(i).
\begin{eqnarray*}
\|\mathcal{I}_k^\alpha (f)\|_{L_k^p(\mathbb{R}^n)} &\le& \left(\int_{\mathbb{R}^n}\left| \frac{1}{\Gamma(\alpha)} \int_0^\infty t^{\alpha-1}\,e^{-t}\, \mathcal{P}_k^t(f)(x)\,dt\right|^p \upsilon_k(x)\,dx \right)^{\frac{1}{p}}\\
&\le& \frac{1}{\Gamma(\alpha)} \int_0^\infty \left( \int_{\mathbb{R}^n}\left|t^{\alpha-1}\,e^{-t}\,\mathcal{P}_k^t(f)(x)\,\right|^p \upsilon_k(x)\,dx \right)^{\frac{1}{p}}dt \\
&\le& c_k\,\|f\|_{L_k^p(\mathbb{R}^n)}.
\end{eqnarray*}
$(ii)$ We obtain the semigroup property  by employing  \eqref{e:5 .1}. This completes the proof.
\end{proof}
\subsection{Wavelet-like transform associated with Poisson integral}
Now, we define the wavelet-like transform associated with the $(k, 1)$-generalized Poisson integral. This transform leads to establishing an inverse formula for the $(k, 1)$-generalized Fourier Flett potentials.
\begin{definition} \label{Wavelet-flett}
    A signed Borel measure $\mu$ on $[0,\infty)$ is said to be a wavelet measure if 
\begin{equation*}
    \mu(\mathbb{R}^+) = \int_0^\infty d\mu(t) = 0\,\, \text{and}\,\, \|\mu\| = \int_0^\infty d|\mu|(t) <\infty.
\end{equation*}
For $f\in L_k^p(\mathbb{R}^n)$, the wavelet-like transform, generated by the wavelet measure $\mu$ and the  $(k, 1)$-generalized Poisson integral $\mathcal{P}_k^t$ associated with $\|x\|\, \Delta_{k}$ is defined as:
\begin{equation*}
    \mathcal{W}^\mu_{k}(f)(x,t) = \int_0^\infty e^{-ty}\, \mathcal{P}_k^{ty}(f)(x)\, d\mu(y), \,\, x\in \mathbb{R}^n,\, t>0.
\end{equation*}
 \end{definition}
\begin{proposition} \label{Wavelet-bounded}
Let $f\in L_k^p(\mathbb{R}^n)$,\, $1\le p\le \infty$. Then the wavelet-like transform $ \mathcal{W}^\mu_{k}$ is bounded on $ L_k^p(\mathbb{R}^n)$.
\end{proposition}
\begin{proof}
Let    $f\in L_k^p(\mathbb{R}^n)$. 
In connection with Minkowski's integral inequality and Theorem \ref{T:4.3}$(i)$, we obtain  
\begin{eqnarray*}
\|\mathcal{W}^\mu_{k}(f)\|_{L_k^p(\mathbb{R}^n)} &\le& \int_0^\infty e^{-ty}\,\|\mathcal{P}_k^{ty}(f)\|_{L_k^p(\mathbb{R}^n)}\, d|\mu|(t)\\
&\le& c_k\, \|f\|_{L_k^p(\mathbb{R}^n)}\, \int_0^\infty e^{-ty}\, d|\mu|(t)\\
&\le& c_k\, \|\mu\|\,\|f\|_{L_k^p(\mathbb{R}^n)}.
\end{eqnarray*}
\end{proof} 
Now, we have all the necessary tools to compute the inversion formula for $(k, 1)$-generalized Fourier Flett potentials.
 The following theorem gives the inversion formula for $(k, 1)$-generalized Fourier Flett potentials. The proof relies on techniques analogous to those used in  \cite[Theorem 3.5]{Negzaoui}.
 \begin{theorem} \label{inversion-Flett}
Let \,  $\mathcal{I}_k^\alpha, \,\, \alpha>0$   be the $(k, 1)$-generalized Fourier Flett potentials of $f$ in $L_k^p(\mathbb{R}^n)$ and $\mathcal{W}_\mu^{t}(f)$  be the wavelet- like transform. If  $\mu$ is a finite Borel measure on $[0,\infty)$ such that 
\begin{eqnarray*}
 \int_1^\infty t^\eta\, d|\mu|(t)&<&\infty\quad\text{for some}\quad \eta>\alpha, \quad \text{and}\\
 \int_0^\infty t^i\, d\mu(t) &=& 0,\,\, i=0,1,\cdots [\alpha]\, (\text{integer part of}\,\,\, \alpha).
\end{eqnarray*}
Then 
\begin{equation} \label{3.1}
    \int_0^\infty y^{-\alpha}\,\mathcal{W}_\mu^{t}(\mathcal{I}_k^\alpha f)(x,y)\, \frac{dy}{y} = C(\alpha,\mu)\,f(x)
\end{equation}
where $C(\theta, \mu)$ is given by  
\[ 
C(\theta, \mu) = \int_0^{\infty} \frac{\Tilde{\mu}(t)}{t^{1+\theta}}dt =  
\begin{cases} 
  \Gamma(-\theta)\int_0^{\infty}s^{\theta}d\mu(s) & \text{ if } \theta \neq 0,1,2,\cdots \\ 
  \frac{(-1)^{\theta+1}}{\theta !} \int_0^{\infty}s^{\theta} \ln{s}d\mu(s) & \text{ if } \theta =0,1,2,\cdots \\  
\end{cases} 
\] The function $\Tilde{\mu}(t)=\int_0^{\infty} e^{-ts}d\mu(s)$ is the Laplace transform of the Borel measure $\mu$.   Moreover, the equation \eqref{3.1} can be interpreted as 
\begin{equation} \label{3.2}
 \underset{\epsilon\rightarrow 0}{\lim}\int_\epsilon^\infty y^{-\alpha}\,\mathcal{W}_\mu^{t}(\mathcal{I}_k^\alpha (f))(x,y)\, \frac{dy}{y} = C(\alpha,\mu)\,f(x).
\end{equation}
The limit \eqref{3.2} exists in the sense of $L_k^p$ norm and point-wise almost every $x\in \mathbb{R}^n$.  If $f\in C_0(\mathbb{R}^n)$, the convergence is uniform.
\end{theorem}
\begin{proof}
Let $f \in L_k^p(\mathbb{R}^n)$,\quad $1 \le p < \infty$. First, we  prove
\begin{equation*} \label{e:5.4}
\mathcal{I}_k^\alpha \circ \mathcal{P}_k^t =  \mathcal{P}_k^t \circ \mathcal{I}_k^\alpha .
\end{equation*}
By using \eqref{e:5 .1} and  Proposition \ref{t:2.3} (iii), we obtain  
\begin{eqnarray*}
\mathcal{F}_k\left( \mathcal{I}_k^\alpha(\mathcal{P}_k^t(f))\right)(\xi) &= &\mathcal{F}_k (\mathcal{P}_k^t(f))(\xi)\, \left(1+\|\xi\|^\frac{1}{2} \right)^{-\alpha} \\
&=& \mathcal{F}_k\left(P_k(t,\cdot)\ast f \right)(\xi)\,\left(1+\|\xi\|^\frac{1}{2} \right)^{-\alpha}\\ 
&=& \mathcal{F}_k\left(P_k(t,\cdot)\right (\xi)\,\left(1+\|\xi\|^\frac{1}{2} \right)^{-\alpha}\, \mathcal{F}_k(f)(\xi) \\
&=& \mathcal{F}_k\left( P_k(t,\cdot) \ast \mathcal{I}_k f\right)(\xi)\\
&=& \mathcal{F}_k(\mathcal{P}_k^t \circ \mathcal{I}_k^\alpha(f))(\xi)\\
\mathcal{I}_k^\alpha \circ \mathcal{P}_k^t &=&  \mathcal{P}_k^t \circ \mathcal{I}_k^\alpha .
\end{eqnarray*}
Now, we consider
\begin{eqnarray*}
\mathcal{W}_\mu^{t}(\mathcal{I}_k^\alpha (f))(x,y) &=& \int_0^\infty e^{-ty} \, \mathcal{P}_k^{ty} \,\mathcal{I}_k^\alpha (f)(t)\, d\mu(t)\\
&=& \int_0^\infty e^{-ty}\,\mathcal{I}_k^\alpha(\mathcal{P}_k^{ty}f)(t)\,d\mu(t)\\
&=&  \frac{1}{\Gamma(\alpha)}\int_0^\infty \int_0^\infty e^{-(ty+s)} \, s^{\alpha-1}\, \mathcal{P}_k^s(\mathcal{P}_k^{ty}(f))(t)\,ds\,d\mu(t)\\
&=& \frac{1}{\Gamma(\alpha)}\int_0^\infty \int_0^\infty e^{-u}\,(u-ty)_+^{\alpha-1}\, \mathcal{P}_k^u(f)(t)\,du\,d\mu(t).
\end{eqnarray*}
For $\epsilon>0$, let us define truncated integral
\begin{eqnarray*}
T_\epsilon f(x) &=& \int_\epsilon^\infty y^{-\alpha-1}\,\mathcal{W}_\mu^{t}(\mathcal{I}_k^\alpha (f))(x,y) \,dy\\
&=& \frac{1}{\Gamma(\alpha)}\int_\epsilon^\infty \int_0^\infty \int_0^\infty y^{-\alpha-1}\,e^{-u}\,(u-ty)_+^{\alpha-1}\, \mathcal{P}_k^u(f)(t)\,du\,d\mu(t)\,dy.
\end{eqnarray*}
By using Fubini's theorem, we obtain
\begin{eqnarray*}
T_\epsilon f(x)&=&\frac{1}{\Gamma(\alpha)}  \int_0^\infty  e^{-u}\,\mathcal{P}_k^u(f)(t) \left( \int_0^\infty \int_\epsilon^\infty  y^{-\alpha-1}\, (u-ty)_+^{\alpha-1}\, dy\, d\mu(t)\right) du\\
&=&\frac{1}{\Gamma(\alpha)}  \int_0^\infty  e^{-u}\,\mathcal{P}_k^u(f)(t) \left( \int_0^{\frac{u}{\epsilon}} \int_\epsilon^{\frac{u}{\mu}} y^{-\alpha-1}\, (u-ty)^{\alpha-1}\, dy\, d\mu(t)\right) du.
\end{eqnarray*}
Replacing $u,y$ with $\epsilon u,\epsilon y$ respectively,
\begin{eqnarray*}
&=&\frac{1}{\Gamma(\alpha)}  \int_0^\infty  e^{- \epsilon u}\,\mathcal{P}_k^{\epsilon u}(f)(t) \left( \int_0^{u}\int_\epsilon^{\frac{u}{t}} y^{-\alpha-1}\, (u-ty)^{\alpha-1}\, dy\, d\mu(t)\right) du.    
\end{eqnarray*}

In view of \cite{Eryigit}, we obtain that 
\begin{eqnarray*}
\int_\epsilon^{\frac{u}{t}} y^{-\alpha-1}\, (u-ty)^{\alpha-1}\, dy = \frac{\Gamma \alpha}{\Gamma (\alpha+1)} \frac{(u-t)^{\alpha}}{u}.
\end{eqnarray*}
We define 
\begin{equation*}
\sigma_{\alpha,\mu}(u) := \frac{1}{\Gamma (\alpha+1)\,u}  \int_0^\infty (u-t)^\alpha \,d\mu(t)
\end{equation*}
and $C({\alpha,\mu}):= \int_0^\infty \sigma_{\alpha,\mu}(u)\, du$.
Thus,
%\begin{eqnarray*}
$T_\epsilon f(x) = \int_0^\infty e^{-\epsilon u} \,\mathcal{P}_k^{\epsilon u} (f)(x)\, \sigma_{\mu, \alpha}(u)\,du.$
%\end{eqnarray*}
Consider,
\begin{eqnarray*}
T_\epsilon f(x) - C(\alpha,\mu) f(x) &=& \int_0^\infty e^{-\epsilon u} \left( \mathcal{P}^{\epsilon u} (f)(x) - f(x)\right)\sigma_{\alpha,\mu}(u)\,du \\&&+f(x) \int_0^\infty  \left(e^{-\epsilon u}-1 \right) \sigma_{\alpha,\mu}(u)\,du.
\end{eqnarray*}
Invoking the Minkowski inequality, we get
\begin{eqnarray*}
    \| T_\epsilon f - C(\alpha,\mu) f\|_{L_k^p(\mathbb{R}^n)} &\le &\int_0^\infty e^{-\epsilon u} \|  \mathcal{P}^{\epsilon u}(f) - f\|_{L_k^p(\mathbb{R}^n)}|\sigma_{\alpha,\mu}(u)|\,du\\&&
    + \|f\|_{L_k^p(\mathbb{R}^n)} \int_0^\infty \left(e^{-\epsilon u}-1 \right) |\sigma_{\alpha,\mu}(u)|\,du.
\end{eqnarray*}
By using the relations  $(i)$ and $(v)$ in Theorem  \ref{T:4.3} and the dominated convergence theorem, we deduce that
\begin{eqnarray*} \label{e:5.5}
\lim_{\epsilon \rightarrow 0}  \| T_\epsilon f - C({\alpha,\mu}) f\|_{L_k^p(\mathbb{R}^n)} = 0,
\end{eqnarray*}
it holds true for every $f \in \mathcal{C}_0(\mathbb{R}^n)$ and the convergence is uniform.\\
The point-wise convergence can be proved similarly as in Theorem \ref{t:4.9} by applying  Theorem \ref{T:4.3} $(v)$. In fact,  
\begin{eqnarray*}
\nonumber \mathop{\underset{\epsilon>0}{\text{sup}}}| T_\epsilon f(x)| &=&   \mathop{\underset{\epsilon>0}{\text{sup}}}\left|  \int_0^\infty e^{-\epsilon u} \,\mathcal{P}_k^{\epsilon u} (f)(x)\, \sigma_{\alpha,\mu}(u)\,du \right| \\
\nonumber &\le& \mathop{\underset{t>0}{\text{sup}}}| \mathcal{P}_k^t(f)(x)|\, \int_0^\infty |\sigma_{\alpha,\mu}(u)|\,du\\
 &\le& C f(x). \label{e:5.6}
\end{eqnarray*}
This completes the proof.
\end{proof}
\section{ The Riesz potential associated with \texorpdfstring{$\beta$}{beta} semigroup} \label{S:4}
The Riesz potential has extensive applications in mathematical physics and geometry. Due to the smoothing effects and ability to describe long-range interactions, the Riesz potential is a key tool in fields like fluid dynamics and electrostatics. \\
In the classical setting,  Sezer and Aliev proposed a characterization of the Riesz potential spaces with the aid of composite wavelet transform \cite{Sezer}.
In this section, we concentrate mainly on the representations of the Riesz potential and the inverse formula for the Riesz potential. For each $\beta >0,$ we introduce a semigroup 
$\{\mathcal{B}_k^{(\beta ,t)}\}_{t>0}$  called $\beta$-semigroup to represent the Riesz potential using the radial kernel  $W_k^{(\beta ,t)}$. The construction of the family $W_k^{(\beta, t)}$ depends on the representations of Riesz potential in terms of heat kernel and Poisson kernel. The section is divided into three parts, the first part deals with some fundamentals of the Riesz potential, the second part consists of a list of properties of the $\beta$-semigroup, and at the end, we deduce an inversion formula for Riesz potential with the help of wavelet-like transform.

\subsection{ The Riesz potential}
 Ivanov introduced the Riesz potential associated with $(k,1)$-generalized Fourier transform \cite{Ivanov1}. 
The Riesz potential for $f\in \mathcal{S}(\mathbb{R})^n$ reads
\begin{eqnarray}\label{Riesz-1}
    \mathbf{R}_k^{\alpha}(f)(x)= \frac{\Gamma(n+2\gamma -1-\alpha)}{\Gamma(\alpha)} \int_{\mathbb{R}^n} \tau_xf(\xi)\|\xi\|^{\alpha-(n+2\gamma-1)}v_k(\xi)d\xi.
\end{eqnarray} 
The Riesz potential can also be written in terms of the heat semi-groups, formulated in \cite{Ivanov1}. We rewrite the integral as 
\begin{align} \label{Riesz-heat}
    \mathbf{R}_k^{\alpha}(f)(x) = \frac{1}{c_k\Gamma(\alpha)} \int_0^{\infty} t^{\alpha-1}H_k^t(f)(x)dt,
\end{align} 

 Apart from the integral representations \eqref{Riesz-1} and \eqref{Riesz-heat}, we define the Riesz potential with the aid of the Poisson kernel as outlined in the following lemma.
\begin{lemma} \label{Riesz-poisson}
Let $f \in \mathcal{S}(\mathbb{R}^n)$ and $0<\alpha< n+2\gamma-1$. Then 
\begin{align*}
    \mathbf{R}_k^{\alpha}(f)(x)= \frac{1}{c_k \Gamma(2\alpha)} \int_{0}^{\infty} t^{2\alpha-1} \mathcal{P}_k^t(f)(x)dt,
\end{align*}
where $\mathcal{P}_k^tf$ is the Poisson semigroup. 

\end{lemma}
\begin{proof}
   We begin the proof by recalling the Poisson semigroup as follows
\begin{eqnarray}
    \int_0^{\infty} t^{2\alpha-1}\mathcal{P}_k^t(f)(x)dt &=& c_k C_{n,k} \int_0^{\infty} t^{2\alpha-1} \int_{\mathbb{R}^n} \frac{t}{(t^2+4\|\xi\|)^{2\gamma+n-1/2}} \tau_xf(\xi)v_k(\xi)d\xi dt \notag \\
    &=& c_k C_{n,k} \int_{\mathbb{R}^n} \tau_xf(\xi) \int_0^{\infty} \frac{t^{2\alpha}}{(t^2+4\|\xi\|)^{2\gamma+n-1/2}} dt v_k(\xi)d\xi \label{Riesz-Poisson-1}.
\end{eqnarray}
The substitution $u=\frac{t}{2\|\xi\|^{1/2}}$ will simplify the inner integral as
\begin{eqnarray}
    \int_0^{\infty} \frac{t^{2\alpha}}{(t^2+4\|\xi\|)^{2\gamma+n-1/2}}dt=(4\|\xi\|)^{\alpha-(2\gamma+n-1)}\int_0^{\infty} \frac{u^{2\alpha}}{(1+u^2)^{2\gamma+n-1/2}}du.\label{Riesz-Poisson-2}
\end{eqnarray}
The integral in the right-hand side of \eqref{Riesz-Poisson-2} boils down to the Beta function by the substitution $u=\tan(\theta),$ 
\begin{align}
    \int_0^{\infty} \frac{u^{2\alpha}}{(1+u^2)^{2\gamma+n-1/2}}du &= \frac{1}{2}\frac{\Gamma(2\gamma+n-\alpha-1) \Gamma(\alpha+1/2)}{\Gamma(2\gamma+n-1/2)}.\label{Riesz-Poisson-3} 
\end{align} 
Thus, the identity follows from \eqref{Riesz-Poisson-1}, \eqref{Riesz-Poisson-2}, \eqref{Riesz-Poisson-3}, and basic properties of Gamma function.
\end{proof}
The boundedness of the Riesz potential 
 for $0<\alpha<(n+2\gamma-1)$ and $\alpha= (n+2\gamma)(\frac{1}{p}-\frac{1}{q}),$ with $1<p<q<\infty$ can also be viewed in the work of Ivanov.

\subsection{$\beta$-semigroup and its properties}
Following the classical theory of Riesz potentials \cite{Aliev-2008}, we define a generalized semigroup generated by the kernel $W_k^{(\beta,t)}$. For each $\beta >0,$ the function $W_k^{(\beta,t)}$ is
\begin{align*}
    W_k^{(\beta,t)}(\xi)= c_k\int_{\mathbb{R}^n} e^{-t\|y\|^{\beta/2}} B_k(\xi,y)v_k(y)dy.
\end{align*}
Then the $\beta$-semigroup generated by the kernel $W_k^{\beta,t}(\xi)$ is defined as
 \begin{align}\label{b-semigroup}
     \mathcal{B}_k^{(\beta,t)}(f)(x) = c_k\int_{\mathbb{R}^n}\tau_x W_k^{(\beta, t)} (\xi)f(\xi)v_k(\xi)d\xi.
 \end{align}
For $\beta=1,$ the $1$-semigroup  $ \mathcal{B}_k^{(1,t)}$ 
  corresponds to the Poisson semigroup, and for $\beta=2,$ the $2$-semigroup $ \mathcal{B}_k^{(2,t)}$ is the heat semigroup. For other values of $\beta,$ the explicit expansion is not straightforward. We can interpret 
 $\mathcal{B}_k^{(\beta,t)}$ as a generalization of these semigroups; thus, it shares similar properties with heat and Poisson semigroups. In the forthcoming proposition, we detail the properties of the kernel $W_k^{(\beta, t)}$ and semi-group $ \mathcal{B}_k^{(\beta,t)},$ with proof techniques extending those from the classical setting and Dunkl setting \cite{Aliev-2008, Rejeb, Verma-25}. 

\begin{proposition}
      Suppose $ 0<\beta <\infty, \, t>0$ and $\xi \in \mathbb{R}^n$. Then 
      \begin{enumerate}[$(i)$]
          \item For any $\lambda >0$, 
          \begin{align*}
              W_k^{(\beta,  \lambda t)}(\lambda^{2/\beta}\xi) = \lambda^{-2(n+2\gamma-1)/\beta}W_k^{(\beta, t)}(\xi).  
          \end{align*} 
          \item  For $0<\beta \le 2, \quad W_k^{(\beta, t)}(\xi)$ is positive.\\
          \item For  any $\beta$ of the form $\beta=4m$, where $m\in \mathbb{N},$ $W_k^{(\beta, t)}(\xi)$ is rapidly decreasing as $\|\xi\|\rightarrow \infty.$ Moreover, for any $\beta>0$ and $t>0,$
          \begin{align*}
              \lim_{\|\xi\| \rightarrow \infty} \|\xi\|^{n+2\gamma-1+\beta/2} \left|W_k^{(\beta, t)}(\xi)\right|= \frac{2^{n+2\gamma-2+\beta}\beta}{\pi} \sin \left(\frac{\beta\pi}{2}\right) \Gamma \left( n+2\gamma-2+\frac{\beta+1}{2}\right) \Gamma\left(\frac{\beta}{2}\right).
          \end{align*}
          \item For $0<\beta <\infty$ and $t>0,$ 
          \begin{align*}
               \int_{\mathbb{R}^n} W_k^{(\beta, t)}(\xi)v_k(\xi)d\xi =1.
          \end{align*}
          \item For $1\le p \le \infty, \, f \in L_k^p(\mathbb{R}^n)$ and for all $t>0,$ we have
          \begin{align*}
              \| \mathcal{B}_k^{(\beta,t)}(f) \|_{L_k^p(\mathbb{R}^n)} &\le C_{\beta} \|f\|_{L_k^p(\mathbb{R}^n)},\\
          \text{where } C_{\beta}= c_k \int_{\mathbb{R}^n} |W_k^{(\beta,1)}(\xi)|v_k(\xi)d\xi <\infty. & \text{ If } 0< \beta \le 2, \text{ then } C_{\beta}=c_k.
          \end{align*}
          \item Let $ f \in L_k^p(\mathbb{R}^n)$  for $p \in [1,\infty)$ and $\beta >0$. Then  
          \begin{align*}
              \sup_{t>0}|\mathcal{B}_k^{(\beta,t)}(f)(x)| \le c\mathcal{M}_kf(x),
          \end{align*}
          where $\mathcal{M}_kf$ is the Maximal function associated with $(k,1)$-generalized Fourier transform.\\ 
          \item  For $ f \in L_k^p(\mathbb{R}^n),\, 1\le p < \infty, $ 
          \begin{align*}
              \sup_{x\in \mathbb{R}^n}|\mathcal{B}_k^{(\beta,t)}(f)(x)|\le \Tilde{c}t^{-\frac{2}{p\beta}(n+2\gamma-1)}\|f\|_{L_k^p(\mathbb{R}^n)}, 
          \end{align*} where $\Tilde{c}=c_k \left(W_k^{(\beta,1)}(0)\right)^{1/p} \|W_k^{(\beta,1)}\|_{L_k^1(\mathbb{R}^n)}^{1/q}$
          \item For a fixed $\beta,$ the family $\mathcal{B}_k^{(\beta,t)}\circ\mathcal{B}_k^{(\beta,\tau)}=\mathcal{B}_k^{(\beta,t+\tau)}$. for all $t, \tau >0.$\\
          \item For $f \in L_k^p(\mathbb{R}^n),$ where $1\le p \le \infty$
          \begin{align*}
              \lim_{t\rightarrow 0} \mathcal{B}_k^{(\beta,t)}(f)(x)= f(x) . 
          \end{align*}
      \end{enumerate}
\end{proposition}
\begin{proof}
    \begin{enumerate}[$(i)$]
        \item   In accordance with the definition of $W_k^{(\beta,t)}(y)$ we have 
        \begin{align*}
            W_k^{(\beta, \lambda t)}(\lambda^{2/\beta}y) &= \mathcal{F}_k(e^{-\lambda t\|\cdot\|^{\beta/2}})(\lambda^{2/\beta}y)\\
            & = c_k \int_{\mathbb{R}^n} e^{-\lambda t\|\xi\|^{\beta/2}}B_k(\lambda^{2/\beta}y, \xi)v_k(\xi) d\xi\\
            & = \lambda^{-2(n+2\gamma-1)/\beta}
            c_k \int_{\mathbb{R}^n} e^{-t\|\xi\|^{\beta/2}} B_k(y, \xi)v_k(\xi )d\xi\\
            & = \lambda^{-2(n+2\gamma-1)/\beta}  W_k^{(\beta, t)}(y).   
        \end{align*}
        \item For $\beta=1$ and $\beta=2,$ the positivity is evident since both the Poisson kernel and heat kernel are positive. Consider the case where $0<\beta<2 $. According to Bernstein's theorem \cite{Feller}, there exists a non-negative finite measure $\eta_{\beta}$ on $[0,\infty)$ such that, 
        \begin{align}
            e^{-w^{\beta/2}}= \int_{0}^{\infty} e^{-uw}d\eta_{\beta}(u) \text{ \quad for $w \ge 0.$} \label{Property-2}
        \end{align} Substituting $w=\|t^{2/\beta}\xi\|$ in \eqref{Property-2}  $W_k^{(\beta, t)}(y)$ becomes 
        \begin{align*}
            W_k^{(\beta, t)}(y) = & c_k \int_{\mathbb{R}^n}\int_0^{\infty} e^{-ut^{2/\beta}\|\xi\|} B_k(y,\xi) d\eta_{\beta}(u)v_k(\xi) d\xi\\
            = & \int_0^{\infty} \left( c_k \int_{\mathbb{R}^n} e^{-ut^{2/\beta}\|\xi\|} B_k(y,\xi)v_k(\xi) d\xi \right) d\eta_{\beta}(u).
        \end{align*}
Using the Proposition \ref{p:3.1} we have 
        \begin{align}
          W_k^{(\beta, t)}(y) = & \int_0^{\infty}  \frac{1}{(ut^{2/\beta})^{n+2\gamma-1}}e^{-\frac{\|y\|}{ut^{2/\beta}}}d\eta_{\beta}(u) <\, \infty. \label{Property-2.1}
        \end{align}
        Thus, positivity of $W_k^{(\beta, t)}$ follows from \eqref{Property-2.1}.

        \item Let $m\in \mathbb{N}$ and $\beta = 4m$. Then the function $W_k^{(\beta, t)}(y)$ is the $(k,1)$-generalized Fourier transform of the Schwartz class function $e^{-t\|\xi\|^{2m}}$, so the rapidly decreasing nature of $W_k^{(\beta, t)}$ is directly followed by  \cite{Grobachev}.
        In the general context, the decay property of $W_k^{(\beta, t)}$ can be deduced from the classical Fourier setting \cite{Aliev-2008}. 

\item The property $(iii)$ possess the integrability of the function $W_k^{(\beta,t)}$, along with the involutory property of $\mathcal{F}_k$ we have
\begin{align*}
    \int_0^{\infty} W_k^{(\beta, t)}(\xi) v_k(\xi)d\xi &=  e^{-t\|0\|^{\beta/2}}=1.
\end{align*}
\item The $\beta$-semigroup as defined in \eqref{b-semigroup}, can also be viewed as a convolution product of the kernel $ W_k^{(\beta,t)}$ and the function $f$, using the boundedness property of convolution product (discussed in Proposition \ref{t:2.3}), we have
\begin{align*}
     \| \mathcal{B}_k^{(\beta,t)}(f) \|_{L_k^p(\mathbb{R}^n)} \leq c_k  \|f\|_{L_k^p(\mathbb{R}^n)}  \|W_k^{(\beta,t)}\|_{L_k^1(\mathbb{R}^n)} \text{ for $f \in L_k^p(\mathbb{R}^n)$.}  
\end{align*}In view of property $(ii)$ and $(iv)$,  for $0< \beta \le 2,$ the constant $C(\beta)=c_k;$ and the general setting, we consider the integral and followed by the change of variables.
\begin{align*}
    \int_{\mathbb{R}^n} |W_k^{(\beta,t)}(\xi)|v_k(\xi)d\xi &=  \int_{\mathbb{R}^n} \big|  \int_{\mathbb{R}^n}  e^{-t\|y\|^{\beta/2}} B_k(\xi, y) v_k(y)dy  \big| v_k(\xi)d\xi\\
    & = \int_{\mathbb{R}^n} \big|  \int_{\mathbb{R}^n}  e^{-\|y\|^{\beta/2}} B_k(\xi, y) v_k(y)dy  \big| v_k(\xi)d\xi\\
    &= \|W_k^{(\beta,1)}\|_{L_k^1(\mathbb{R}^n)} < \infty, \text{ true for all $t>0$}.
    \end{align*}
    \item We use the  Theorem \ref{t:4.9} for proving the property $(vi)$.  Let $\psi_t$ be  a dilated family, defined as  $\psi_t(y)= t^{n+2\gamma-1} W_k^{(\beta,1)}(ty)$, then 
    \begin{align*}
        \sup_{t>0}|\mathcal{B}_k^{(\beta,t)}(f)(x)| =  \sup_{t>0}|W_k^{(\beta,t)}\ast f (x) | = \sup_{t>0}| \psi_t \ast f (x)| \le c\mathcal{M}_kf(x).
    \end{align*}
    \item  Let
    \begin{align*}
        G_k^{\beta}(y)= c_k \int_{\mathbb{R}^n} e^{-\|\xi\|^{\beta/2}} B_k(y,\xi)v_k(\xi)d\xi. 
    \end{align*}Then, 
    \begin{align} \label{G-w-kernel}
        W_k^{(\beta,t)}(y) = t^{-\frac{2}{\beta}(n+2\gamma-1)}G_k^{\beta}(t^{-2/\beta}y).
    \end{align}
     Now we consider the point-wise bound of  $\mathcal{B}_k^{(\beta,t)}(f)(x)$ 
    \begin{align}\label{Property-7}
       \big| \mathcal{B}_k^{(\beta,t)}(f)(x) \big| \le c_k \int_{\mathbb{R}^n} \big| \tau_x W_k^{(\beta,t)}(y)f(y)\big|v_k(y)dy. 
    \end{align}
    In sight of H\"older's inequality, we get 
    \begin{align*}
        \int_{\mathbb{R}^n}& \big| \tau_x W_k^{(\beta,t)}(y)f(y)\big|v_k(y)dy
        \\ & \le \left( \int_{\mathbb{R}^n} |f(y)|^p|\tau_xW_k^{(\beta,t)}(y)| v_k(y)dy
        \right)^{1/p} \left( \int_{\mathbb{R}^n} |\tau_xW_k^{(\beta,t)}(y)|v_k(y)dy \right)^{1/q}. 
    \end{align*} Invoking the properties of the translation and the identity \eqref{G-w-kernel}, the above inequality reduces to
    \begin{align}
          & \int_{\mathbb{R}^n} \big| \tau_x W_k^{(\beta,t)}(y)f(y)\big|v_k(y)dy\notag \\
          & \le  c_k^2 \left( \int_{\mathbb{R}^n} |f(y)|^pv_k(y)dy \, \int_{\mathbb{R}^n} e^{-t\|\xi\|^{\beta/2}} v_k(\xi)d\xi
          \right)^{1/p}   \left( \int_{\mathbb{R}^n} |W_k^{(\beta,t)}(y)|v_k(y)dy \right)^{1/q}\notag \\
         & \le c_k \|f\|_{L_k^p(\mathbb{R}^n)}  t^{-\frac{2}{p\beta}(n+2\gamma-1)}
         \left(G_k^{\beta}(0)\right)^{1/p} \|W_k^{(\beta,1)}\|_{L_k^1(\mathbb{R}^n)}^{1/q}. \label{Property-7-2}
    \end{align} Thus we have the result from  \eqref{Property-7} and   \eqref{Property-7-2} with the constant\\ $\Tilde{c}=c_k \left(G_k^{\beta}(0)\right)^{1/p} \|W_k^{(\beta,1)}\|_{L_k^1(\mathbb{R}^n)}^{1/q}.$
    \item  The semi-group property will directly follow by applying the $(k,1)$-generalized Fourier transform.
    \item The proof of the property $(ix)$, follows as in the case of the Poisson semigroups (Theorem \ref{T:4.3}) by substituting $W_k^{(\beta, t)}$ and $t^{\frac{2}{\beta}}$ instead of poison kernel and $t^2$ respectively.
\end{enumerate}
\end{proof}
\subsection{ Inversion of Riesz potential} We establish an inversion formula for the Riesz potential in accordance with the wavelet-like transform  
With the aid of $\beta$-semigroup $\mathcal{B}_k^{(\beta,t)},$ we represent the Riesz-potential in the following theorem.  In cases like $\beta= 1$ and $2,$ the identity is established in Lemma \ref{Riesz-poisson} and identity \eqref{Riesz-heat}, respectively.

\begin{theorem}
    Let $0<\alpha< n+2\gamma-1, \, f\in L_k^p(\mathbb{R}^n)$ for $1< p < \frac{n+2\gamma-1}{\alpha}.$ The Riesz-potential of $f$ possesses the integral representation 
    \begin{align*}
        \mathbf{R}_k^{\alpha}(f)(y)= \frac{c_k^{-1}}{\Gamma(2\alpha/\beta)}\int_0^{\infty} t^{\frac{2\alpha}{\beta}-1}\mathcal{B}_k^{(\beta,t)}(f)(y)dt.
    \end{align*}
\end{theorem}
\begin{proof}
      We are proving the result by considering the $(k,1)$-generalized Fourier transform of the function $ \phi_f(y)=  \int_0^{\infty} t^{\frac{2\alpha}{\beta}-1}  \mathcal{B}_k^{(\beta,t)}(f)(y)dt,$
     where $f \in \mathcal{S}(\mathbb{R}^n)$.  
    Thus, \begin{align*}
        \mathcal{F}_k(\phi_f)(\xi)&= c_k\int_{\mathbb{R}^n} \phi_f(y)B_k(\xi,y)v_k(y)dy\\
        & =  \int_0^{\infty}  t^{\frac{2\alpha}{\beta}-1}  \mathcal{F}_k(\mathcal{B}_k^{(\beta,t)}(f))(\xi)dt\\
        & = \mathcal{F}(f)(\xi) \int_{0}^{\infty}  t^{\frac{2\alpha}{\beta}-1} e^{-t\|\xi\|^{\beta/2}}dt.
    \end{align*}
Changing the variable $u=t\|\xi\|^{\beta/2},$ and the identity $\mathcal{F}_k( \mathbf{R}_k^{\alpha}(f))(\xi)=\frac{\mathcal{F}(f)(\xi)}{\|\xi\|^{\alpha}},$ we obtain the result.
\end{proof}
Now we construct a wavelet-like transform by making the use of $\beta$-semigroup $\mathcal{B}_k^{(\beta, t)}$. Let $\mu$ be the signed-Borel measure on $[0,\infty)$ (see Definition \ref{Wavelet-flett}). 

\begin{align}
    \mathcal{W}_k^{\beta}f(y,t) = \int_0^{\infty} \mathcal{B}_k^{(\beta, ts)}(f)(y)d\mu(s).
\end{align}
The wavelet-like transform is well-defined for any function $f\in L_k^p(\mathbb{R}^n)$
\begin{align*}
    \|  \mathcal{W}_k^{\beta}f(\cdot,t)\|_{L_k^p(\mathbb{R}^n)} \le C_{\beta}\,\|\mu\|\|f\|_{L_k^p(\mathbb{R}^n)}.
\end{align*}In the following theorem, we derive an inversion formula for the Riesz potential using the wavelet-like transform. The proof follows a similar approach to the classical setting, as presented in \cite{Sezer}.

\begin{theorem}\label{Riesz-inversion}
Let $0< \alpha< n+2\gamma-1, \beta>0$ and $f\in L_k^p(\mathbb{R}^n)$ for $1< p < \frac{n+2\gamma-1}{\alpha}$. Suppose that $\mu$ is a finite Borel measure, which satisfies the conditions 

\begin{eqnarray*}
 \int_1^\infty t^\eta\, d|\mu|(t)&<&\infty\quad\text{for some}\quad \eta> \frac{2\alpha}{\beta}, \quad \text{and}\\
 \int_0^\infty t^i\, d\mu(t) &=& 0,\,\, i=0,1,\cdots \left[\frac{2\alpha}{\beta}\right]\, \left(\text{integer part of}\,\,\, \frac{2\alpha}{\beta}\right).
\end{eqnarray*}
If $\phi = \mathbf{R}_k^{\alpha}(f)$, then 
\begin{align*}
\int_0^{\infty} \mathcal{W}_k^{\beta}\phi(y,t)\frac{dt}{t^{1+2\alpha/\beta}}= \lim_{\epsilon\rightarrow 0} \int_{\epsilon}^{\infty}\mathcal{W}_k^{\beta}\phi(y,t)\frac{dt}{t^{1+2\alpha/\beta}} = C\left({2\alpha}/{\beta}, \mu \right)f(y),
\end{align*}where $C(\theta, \mu)$ is given in Theorem \ref{inversion-Flett}.
\end{theorem}
\begin{Remark}
    The integral representation of Riesz potential $\mathbf{R}_k^{\alpha}$ for $0<\alpha< n + 2\gamma-1 $,  in terms of $\beta$-semigroup is valid for $L_k^p(\mathbb{R}^n)$ functions for $1 < p < \frac{n+2\gamma-1}{\alpha}.$  The parameter $\beta$ is free to choose in a wide range. If we choose $\beta >2\alpha$, then $\left[\frac{2\alpha}{\beta}\right]=0,$ so the vanishing moment has no particular role in the condition, our measure itself satisfying the required condition. However we need $\int_1^{\infty}s\,d|\mu|(s)< \infty$.  
\end{Remark}
Now we introduce a space associated with the Riesz potential in the $(k,1)$-generalized Fourier setting.

\textbf{Riesz potential space:} The Riesz potential space of order $\alpha$, where $0<\alpha<n + 2\gamma-1$ is the collection of all functions which are the image of $L_k^p(\mathbb{R}^n)$ under the map $\mathbf{R}_k^{\alpha}.$
\begin{align*}
    \mathbf{R}_k^{\alpha}(L_k^p(\mathbb{R}^n))= \{ \phi: \phi = \mathbf{R}_k^{\alpha}(f), \, f \in L_k^p(\mathbb{R}^n) \},
\end{align*} where $1< p < \frac{n+2\gamma-1}{\alpha}$. The norm on the space 
$ \mathbf{R}_k^{\alpha}(L_k^p(\mathbb{R}^n)),$ is given by the relation $\|\phi\|_{ \mathbf{R}_k^{\alpha}(L_k^p(\mathbb{R}^n))}=\|f\|_{L_k^p(\mathbb{R}^n)}$ it induces a structure of complete normed linear space. 

\section{Bi-parametric potential type spaces} \label{S:5}
Recently, Aliev explored bi-parametric potential spaces for the classical Fourier transform \cite{Aliev}. This section extends his approach to the 
$(k,1)$-generalized Fourier transform.  Now, we generalize both the Bessel and Flett potentials in terms of the 
 $\beta$-semigroup. For 
$\beta\in (0,\infty)$, we introduce the corresponding bi-parametric potentials and define the associated potential spaces.
\begin{definition}
 Let $f\in L_k^p(\mathbb{R}),\,\,1\le p \le \infty$   and for any $\alpha, \beta>0$. We define the bi- parametric potentials $\mathfrak{J}_k^{(\alpha,\beta)}$ of order $\alpha$ of $f$  by
\begin{equation*}
 \mathfrak{J}_k^{(\alpha,\beta)}(f)(x) = \frac{1}{\Gamma(\alpha/\beta)} \int_0^\infty t^{\frac{\alpha}{\beta}-1}\,e^{-t}\,\mathcal{B}_k^{(\beta,t)}(f)(x)\,dt.  
\end{equation*}
\end{definition}

\begin{theorem}
Let $\alpha, \beta >0$ and $f\in L_k^p(\mathbb{R})$, $1\le p <\infty$  and $f\in C_0(\mathbb{R})$  for $p=\infty$. Then 
\begin{itemize}
\item [$(i)$] $ \mathfrak{J}_k^{(\alpha,\beta)}$ is well defined on $L_k^p(\mathbb{R}^n)$ and it is a bounded operator,
\begin{equation*}
\| \mathfrak{J}_k^{(\alpha,\beta)}(f)\|_{L_k^p(\mathbb{R}^n)} \le c_k\, \|f\|_{L_k^p(\mathbb{R}^n)}.
        \end{equation*}
\item[$(ii)$] The family $( \mathfrak{J}_k^{(\alpha,\beta)})_{\alpha\ge0}$ satisfies the semigroup property $ \mathfrak{J}_k^{(\alpha_1,\beta)}\circ \, \mathfrak{J}_k^{(\alpha_2,\beta)} = \mathfrak{F}_k^{(\alpha_1+\alpha_2, \beta)}$.
\end{itemize}
\end{theorem}
\begin{proof}
The proof follows the same approach used in the proof of Theorem \ref{t:3.10}.  
\end{proof}
\begin{Remark} We observe the following
    \begin{itemize}
        \item [$(i)$]If $\beta=1$, then $\mathfrak{J}_k^{(\alpha,\beta)} $ reduces to Flett potential $\mathcal{I}_k^{\alpha}$.\\
        \item[$(ii)$] If $\beta=2$, then $\mathfrak{J}_k^{(\alpha,\beta)}$ reduces to Bessel potential $ \mathbf{I}_k^{\alpha}$.
    \end{itemize}
\end{Remark}
We define the bi-parametric wavelet-like transform below.
\begin{definition}
For $\alpha,\beta>0$ and $f\in L_k^p(\mathbb{R}^n)$, the bi-parametric wavelet-like transform generated by the wavelet measure $\mu$ is defined by 
\begin{equation*}
    \mathcal{A}^{\mu,\beta}_{k}(f)(x,t) = \int_0^\infty e^{-ty}\, \mathcal{B}_k^{(\beta,ty)}f(x)\, d\mu(y), \,\,   t>0.
\end{equation*}
\end{definition}
\begin{proposition} 
Let $\alpha,\beta>0$ and $f\in L_k^p(\mathbb{R}^n)$,\, $1\le p\le \infty$. Then the wavelet-like transform $  \mathcal{A}^{\mu,\beta}_{k}$ is bounded on $ L_k^p(\mathbb{R}^n)$.
\end{proposition}
 \begin{theorem} 
Let \,  $\mathfrak{J}_k^{(\alpha,\beta)}, \,\, \alpha,\beta>0$   be the bi-parameter  potential of $f$ in $L_k^p(\mathbb{R})$ and $\mathcal{A}^{\mu,\beta}_{k}(f)$  be the 
 bi-parameter wavelet-like transform. If  $\mu$ is a finite Borel measure on $[0,\infty)$ such that 
\begin{eqnarray*}
 \int_1^\infty t^\eta\, d|\mu|(t)&<&\infty\quad\text{for some}\quad \eta> \frac{\alpha}{\beta},\\
 \int_0^\infty t^i\, d\mu(t) &=& 0,\,\, i=0,1,\cdots \left[\frac{\alpha}{\beta}\right]\,\left(\text{integer part of}\,\,\, \frac{\alpha}{\beta}\right).
\end{eqnarray*}
Then 
\begin{equation} \label{e.3.1}
    \int_0^\infty y^{-(\frac{\alpha}{\beta}+1)}\,\mathcal{A}^{\mu,\beta}_{k}(f)(x,y)\, dy = C(\alpha/\beta, \mu)\,f(x),
\end{equation}
where the constant is given in Theorem  \ref{inversion-Flett}. Also, the equation \eqref{e.3.1} interpreted as 
\begin{equation} \label{e.3.2}
 \underset{\epsilon\rightarrow 0}{\lim}\int_0^\infty y^{-(\frac{\alpha}{\beta}+1)}\,\mathcal{A}^{\mu,\beta}_{k}(f)(x,y)\, dy = C(\mu,\alpha/\beta)\,f(x).
\end{equation}
The limit \eqref{e.3.2} exists in the sense of $L_k^p$ norm and point-wise almost every $x\in \mathbb{R}^n$.  If $f\in \mathcal{C}_0(\mathbb{R}^n)$, the convergence is uniform.
\end{theorem}
\begin{proof}
By using the same proof techniques of Theorem \ref{inversion-Flett}, we can prove this theorem.    
\end{proof}
\begin{note}
   If $\beta=1$, then $ \mathcal{A}^{\mu,1}_{k} $ reduces to Flett potential wavelet-like transform $\mathcal{W}^\mu_{k}$. Similarly, for $\beta=2$,  $\mathcal{A}^{\mu,2}_{k} $  will be a wavelet-like Bessel potential transform.  
\end{note}
Next, we define the bi-parametric potential space which generalizes the Flett potential space.
\begin{definition}
Let $\alpha, \beta>0$ and $1 <p<\infty$. Then the bi-parametric potential space is the image of $L_k^p(\mathbb{R}^n)$ under the operator $\mathcal{J}_k^{\alpha,\beta}$
\begin{equation*}
\mathfrak{H}^{\alpha,\beta,p}_k = \left\{ f : \mathfrak{J}_k^{(\alpha,\beta)}(\phi) = f, \,\, \text{for some}\,\,\phi \in L_k^p(\mathbb{R}^n) \right\}.
\end{equation*}
\end{definition}
We define the norm on $\mathfrak{H}^{\alpha,\beta,p}_k$ such that $ \|f\|_{\mathfrak{H}^{\alpha,\beta,p}_k} = \|\phi\|_{L_k^p(\mathbb{R}^n)}$. Certainly, we conclude that $\mathfrak{F}^{\alpha,p}_k$ is a Banach space.

\begin{note}
If $\beta=1,2$,  then $\mathfrak{H}^{\alpha,1,p}_k, \mathfrak{H}^{\alpha,2,p}_k$ are called the Flett and Bessel potential spaces, respectively.  
\end{note}

\textbf{Acknowledgment:}
The authors extend their gratitude for the financial support provided by the UGC (221610147795) for the first author (AP) and the ANRF with the reference number SUR/2022/005678, for the third author (SKV). Their funding and resources were instrumental in the completion of this work. \\  
\\
\textbf{Declarations:}\\ \\
\textbf{Conflict of interest:}
We would like to declare that we do not have any conflict of interest.\\ \\
\textbf{Data availability:}
No data source is needed.

\bibliographystyle{abbrvnat} 
 
\end{document}